\documentclass[12pt]{amsart}
\usepackage{amssymb}
\theoremstyle{definition}

\newcommand{\Z}{\mathbb{Z}}
\newcommand{\ol}{\overline}
\newcommand{\ul}{\underline}

\newtheorem{theorem}{Theorem}[section]
\newtheorem{lemma}[theorem]{Lemma}
\newtheorem{example}[theorem]{Example}
\newtheorem{definition}[theorem]{Definition}
\newtheorem{proposition}[theorem]{Proposition}

\newtheorem{corollary}[theorem]{Corollary}

\begin{document}
\title{Idempotents and structures of rings}

\author{P.N. \'Anh}
\address{R\'enyi Institute of Mathematics, Hungarian Academy of Sciences,
1364 Budapest, Pf.~127, Hungary}
\email{anh@renyi.hu}
\author{G.F. Birkenmeier}
\address{Department of Mathematics, University of Louisiana at Lafayette, Lafayette, LA 70504-1010, USA}
\email{gfb1127@louisiana.edu}
\author{L. van Wyk}
\address{Department of Mathematical Sciences, Stellenbosch University,
P/Bag X1, \hfill\break Matieland~7602, Stellenbosch, South Africa }
\email{LvW@sun.ac.za}
\thanks{Corresponding author: G.F. Birkenmeier}

\begin{abstract} Recall that an $n$-by-$n$ generalized matrix ring is defined in terms of sets of rings $\{R_i\}_{i=1}^n, \ (R_i, R_j)$-bimodules $\{M_{ij}\}$, and bimodule homomorphisms $\theta_{ijk} : M_{ij} \otimes_{R_j} M_{jk} \rightarrow M_{ik}$, where the set of diagonal matrix units $\{E_{ii}\}$ form a complete set of orthogonal idempotents.  Moreover, an arbitrary ring with a complete set of orthogonal idempotents $\{e_i\}_{i=1}^n$ has a Peirce decomposition which can be arranged into an $n$-by-$n$ generalized matrix ring $R^\pi$ which is isomorphic to $R$.  In this paper, we focus on the subclass $\mathcal{T}_n$ of $n$-by-$n$ generalized matrix rings with $\theta_{iji} = 0$ for $i\ne j$.  $\mathcal{T}_n$ contains all upper and all lower generalized triangular matrix rings.  The triviality of the bimodule homomorphisms motivates the introduction of three new types of idempotents called the inner Peirce, outer Peirce, and Peirce trivial idempotents.  These idempotents are our main tools and are used to characterize $\mathcal{T}_n$ and define a new class of rings called the $n$-Peirce rings.  If $R$ is an $n$-Pierce ring, then there is a certain complete set of orthogonal idempotents $\{e_i\}_{i=1}^n$ such that $R^\pi\in\mathcal{T}_n$.  We show that every $n$-by-$n$ generalized matrix ring $R$ contains a subring $S$ which is maximal with respect to being in $\mathcal{T}_n$ and $S$ is essential in $R$ as an $(S, S)$-bisubmodule of $R$.  This allows for a useful transfer of information between $R$ and $S$.  Also, we show that any ring is either an $n$-Peirce ring or for each $k > 1$ there is a complete set of orthogonal idempotents $\{e_i\}_{i=1}^k$ such that $R^\pi\in\mathcal{T}_k$.  Examples are provided to illustrate and delimit our results.
\end{abstract}

\maketitle

\section*{Introduction}
\label{notions}
\smallskip
Throughout this paper all rings are associative with a unity and modules are unital unless explicitly indicated otherwise.

Given a complete set of orthogonal idempotents, $\{e_i\}_{i=1}^n$, of a ring $R$, we can form  a group direct sum,
$$R=e_1Re_1\oplus\cdots\oplus e_1Re_n\oplus e_2Re_1\oplus\cdots\oplus e_2Re_n \oplus \cdots \oplus e_nRe_1\oplus\cdots\oplus e_nRe_n,$$
called the Peirce decomposition of $R$. This decomposition can be arranged into an $n$-by-$n$ square array, called $R^\pi$, with

$$R^\pi= \left[ \begin{array}{ccccc}
e_1Re_1 & e_1Re_2 & \cdots & e_1Re_n\\
\\
e_2Re_1 & e_2Re_2 & \ddots & \vdots\\
\\
\vdots & \ddots & \ddots & e_{n-1}Re_n \\
\\
e_nRe_1 & \cdots & e_nRe_{n-1} &  e_nRe_n \end{array}\right].$$

\noindent The array $R^\pi$ forms a ring, where addition is componentwise and multiplication is the usual row-column matrix multiplication. Moreover, there is a ring isomorphism $h: R \rightarrow R^\pi$ defined by $h(x) =[e_i x e_j]$ for all $x\in R$. Observe that the $e_iRe_i$ are rings with unity and the $e_iRe_j$ are $(e_iRe_i, e_jRe_j)$-bimodules. Note that the bimodule product $e_iRe_j \cdot e_jRe_k$, arising in the row-column multiplication, may be thought of as a bimodule homomorphism $\theta_{ijk}: e_iRe_j \otimes_{e_jRe_j} e_jRe_k \to e_iRe_k$ determined by the multiplication of $R$.

The above discussion motivates the following well known definition:

\medskip

\centerline{an {\it $n$-by-$n$ generalized (or formal) matrix ring $R$} is a square array}

$$R = \left[ \begin{array}{cccc}
R_1 & M_{12} & \cdots & M_{1n}\\
\\
M_{21} & R_{2} & \ddots & \vdots \\
\\
\vdots & \ddots & \ddots & M_{n-1,n} \\
\\
M_{n1} & \cdots & M_{n,n-1} & R_n \end{array}\right]$$

\noindent where each $R_i$ is a ring, each $M_{ij}$ is an $(R_i, R_j)$-bimodule and there exist $(R_i, R_k)$-bimodule homomorphisms $\theta_{ijk} : M_{ij} \otimes_{R_j} M_{jk} \rightarrow M_{ik}$ for all $i, j, k = 1, \ldots, n$ (with $M_{ii} = R_i$). For $m_{ij} \in M_{ij}$ and $m_{jk} \in M_{jk}$, $m_{ij}m_{jk}$ denotes $\theta_{ijk}(m_{ij} \otimes m_{jk})$. The homomorphisms $\theta_{ijk}$ must satisfy the associativity relation: $(m_{ij}m_{jk}) m_{k\ell} = m_{ij}(m_{jk} m_{k\ell})$ for all $m_{ij} \in M_{ij},  \ m_{jk} \in M_{jk}, \ m_{k\ell} \in M_{k\ell}$ and all $i,j,k, \ell = 1, \ldots, n$.  Observe that $\theta_{iii}$ is determined by the ring multiplication in $R_i$, while $\theta_{ijj}$ and $\theta_{jjk}$ are determined by the bimodule scalar multiplications. Further information on generalized matrix rings can be found in [KT].

With these conditions, addition on~$R$ is componentwise and multiplication on~$R$ is row-column matrix multiplication. 
A Morita context is a $2$-by-$2$ generalized matrix ring. {\it An $n$-by-$n$ generalized upper (lower) triangular matrix ring} is a generalized matrix ring with $M_{ij}=0$ for $j< i$ ($M_{ij}=0$ for $i< j).$  Note that $\{E_{ii} \in R \  \vert \ E_{ii}$ is the matrix with $1 \in R_i$ in the $(i,i)$-position and $0$ elsewhere, $i = 1, \ldots, n\}$ is a complete set $\{E_{ii}\}_{i=1}^n$ of orthogonal idempotents in the above constructed generalized matrix ring $R$. 

The foregoing observations allow us to consider a generalized matrix ring in two ways:

\medskip

(1) given a ring $R$ and a complete set of orthogonal idempotents, $\{e_i\}_{i=1}^n$, then $R^\pi$ is an "internal" representation of $R$ as a generalized matrix ring in terms of substructures of $R$; whereas

(2) given collections $\{R_i\}, \ \{M_{ij}\},$ and $\{\theta_{ijk}\}$, we construct a new ring from these "external" components via the  generalized matrix ring notion.

\medskip

An important problem in the study of  generalized matrix rings is: given a collection of rings $\{R_i \ \vert \ i=1, \ldots,n\}$ and bimodules $\{M_{ij} \ \vert \ i,j=1, \ldots,n, \ i\ne j,$ and each $M_{ij}$ is an $(R_i,R_j)$-bimodule$\}$ determine the $\theta_{iji} \ (i\ne j)$ and the $\theta_{ijk}$ ($i,j,k$ distinct) to produce an $n$-by-$n$ generalized matrix ring. We can simplify this problem by trivializing the $\theta_{ijk}$ in the following three ways (note that for $n=2$, all three conditions coincide):

\medskip

(I) Define $\theta_{iji}=0$, for  all $i\ne j$.

(II) Define $\theta_{ijk}=0$, for  all $i,j,k$ pairwise distinct.

(III) Define $\theta_{ijk}=0$, for  all $i\ne j$ and $j\ne k$ (I and II combined).

\medskip

Two questions immediately arise:

\medskip

(A) Are there significant examples of generalized matrix rings with trivialized~$\theta_{ijk}$?

(B) How can the theory of generalized matrix rings with trivialized $\theta_{ijk}$ be used to gain insight into the theory of arbitrary generalized matrix rings?

\medskip

In this paper, we consider the class of $n$-by-$n$ ($n>1$) generalized matrix rings satisfying condition (I) (i.e., $\theta_{iji}=0$, for  all $i\ne j$). 

\medskip

\centerline{{\it We denote this class of rings by $\mathcal{T}_n$.}}

\medskip

For each generalized matrix ring~$R$, 

\medskip

\centerline{{\it we use $\ol{R}$ to denote the ring in $\mathcal{T}_n$ which has the same corresponding $R_i, M_{ij}, \theta_{ijk}$}} 

\centerline{{\it as $R$, except that for all $i \neq j$ the homomorphisms~$\theta_{iji}$ are taken to be $0$ in $\ol{R}$.}}

\medskip

Thus $R$ and $\ol{R}$ are the same ring if and only if $R \in \mathcal{T}_n$. Note that the classes of $n$-by-$n$ generalized upper and lower triangular  matrix rings form significant proper subclasses of $\mathcal{T}_n$ (see Question~A). Further examples are provided throughout this paper.

Observe that the triviality of the $\theta_{iji}$ motivates three new types of idempotents which appear in the internal (Peirce decomposition) generalized matrix ring representation of a ring in $\mathcal{T}_2$. For $e=e^2\in R$, 

\medskip

(1) $e$ is {\it inner Peirce trivial} if $eR(1-e)Re=0$;

(2) $e$ is {\it outer Peirce  trivial} if $(1-e)ReR(1-e) =0$;

(3) $e$ is  {\it Peirce trivial} if $e$ is both inner and outer trivial. 

\medskip

In [P] B. Peirce introduced the concept of an idempotent, and so we are naming certain idempotents and rings in this paper in his honor.  These idempotents provide the main tools in our investigations; in particular, they are used to characterize the class $\mathcal{T}_n$ and the class of $n$-Peirce rings.

In Section 1, we develop the basic properties of the inner (outer) Peirce trivial idempotents. Moreover we show that if $R$ is a subring of a ring $T$ and $S$ is the subring of $T$ generated by $R$ and a subset $\mathcal{E}$ of inner or outer Peirce trivial idempotents of $T$, then there is a useful transfer of information between $R$ and~$S$, e.g., $R$ is strongly $\pi$-regular or has classical Krull dimension $n$ if and only if so does $S$ (Theorems~1.13, 1.14 and 1.16).

In Section 2, we begin by showing that, for a ring $R$ with a complete set $\{e_i \}_{i=1}^n$ of orthogonal idempotents, $R^\pi \in \mathcal{T}_n$ if and only if each $e_i$ is inner Peirce trivial~(Theorem~2.2). 

Next, let $R$ be a generalized $n$-by-$n$ matrix ring and take 

$$\mathcal{D}(R) = \{[r_{ij}] \in R \ \vert \ r_{ij} =0 \ {\rm for\ all\ }i \neq j\}$$ 

and 

$$\mathcal{D}(R)^- = \{[r_{ij}] \in R \ \vert \ r_{ii} =0\ {\rm for \ all\ }i =1, \ldots, n\}.$$

\noindent We obtain that if $R \in \mathcal{T}_n$ then $\mathcal{D}(R)^- \trianglelefteq R$ such that $(\mathcal{D}(R)^-)^n=0$ and $\mathcal{D}(R)$ is ring isomorphic to $R/(\mathcal{D}(R)^-)$ (Proposition 2.4). 

The transfer of various ring properties (e.g., semilocal, bounded index, having a polynomial identity) between~$R$ and $\mathcal{D}(R)$ is considered when $R \in \mathcal{T}_n$. In~Theorem~2.12 (one of the main results of the paper) we show that every $n$-by-$n$ generalized matrix ring has  subrings $S$ maximal with respect to being in $\mathcal{T}_n$ such that $S$ is essential in $R$ as an $(S,S)$-bimodule. This fact allows for a two-step transfer of information from $\mathcal{D}(R)$ to~$S$ (Theorem 1.16) and from $S$ to $R$ (Theorem~2.12 and Corollary 2.14). 

In the remainder of this section, we introduce the notion of an ideal extending ring and use Theorem 2.12 and its consequences to show how this notion passes from a ring $A$ to certain generalized matrix rings which are overrings of the $n$-by-$n$ upper triangular matrix ring over $A$. Thus Theorem~2.12 and its corollaries provide answers to Question B. 

The $n$-Peirce rings are introduced and investigated in Section 3. A ring $R$ is a 

\medskip 

\centerline{{\it 1-Peirce ring}} 

\medskip

\noindent if $0$ and $1$ are the only Peirce trivial idempotents in $R$. Inductively, for a natural number $n > 1$, we say a ring $R$ is an 

\medskip

\centerline{{\it $n$-Peirce ring}} 

\medskip 

\noindent if there is a Peirce trivial idempotent $e$ such that $eRe$ is an $m$-Peirce ring for some $1 \leq m < n$ and $(1-e)R(1-e)$ is an $(n-m)$-Peirce ring. 

In Theorem 3.7, we show that an $n$-Peirce generalized matrix ring is in $\mathcal{T}_n$ ($n > 1$) and that if $R$ has a complete set $\{e_i\}^n_{i=1}$of orthogonal idempotents  such that each $e_i Re_i$ is a 1-Peirce ring, then $R$ is a $k$-Peirce ring for some $1 \leq k \leq n$. Example 3.2 shows that the class of $n$-Peirce rings is a proper subclass of $\mathcal{T}_n$ for $n > 1$, and that any $n$-by-$n$ generalized upper triangular matrix ring with prime diagonal rings is an $n$-Peirce ring. The class of $n$-Peirce rings has an advantange over $\mathcal{T}_n$ in that for $n > 1$, an $n$-Peirce ring has a complete set $\{e_i\}^n_{i=1}$ of orthogonal idempotents  such that each $e_iRe_i$ is a 1-Peirce ring. In Theorem 3.11, it is also shown that if $R$ has DCC on $\{ReR \ \vert \ e$ is a Peirce trivial idempotent$\}$,  then $R$ is an $n$-Peirce ring for some~$n$.

As indicated in Definition 1.1, the inner (outer) Peirce trivial idempotents can be defined in a ring without a unity. Hence, many of the results in this paper can be modified to hold in rings without a unity.

\vskip0.7cm

\section*{Notation and Terminology}
\label{notions}

\begin{enumerate}
  \item $R$ is Abelian - means every idempotent is central.
	\item $\mathcal{B}(R), \mathcal{P}(R)$ and $\mathcal J(R)$ denote the central idempotents of $R$, the prime radical of~$R$ and the Jacobson radical of $R$ respectively.
	\item $\mathcal{S}_{\ell}(R) = \{e=e^2\in R \ \vert \ \ Re = eRe\}, \ \mathcal{S}_r(R) = \{e=e^2\in R \ \vert \ eR = eRe\}$.
	\item ${\rm Cen}(R)$ is the center of $R$.
	\item ${\rm U}(R)$ is the group of units of $R$.
	\item $<->_R$ is the subring of $R$ generated by $-$, and $(-)_R$ is the ideal of $R$ generated by $-$.
	\item $X \trianglelefteq R$ means $X$ is an ideal of $R$.
	\item $\ul{r}_A(B)$ and $\ul{\ell}_A(B)$ denote the right and left annihilator of $B$ in $A$, respectively.
	\item $\Z$ and $\Z_n$ denote the ring of integers and the ring of integers modulo $n$, respectively.
	\item $\Z^+$ means the positive integers.	
\end{enumerate}

\vskip0.5cm

\section{Basic Properties of Peirce trivial Idempotents}

\begin{definition}
\label{semitrivial} 
Let $R$ be a ring, not necessarily with a unity, and let $e=e^2\in R$. We say $e$ is \emph{inner Peirce trivial} (respectively, \emph{outer Peirce trivial}) if
 $exye=exeye$ (respectively, $xey+exeye=xeye+exey$) for all $x, y\in R$. If $e$ is both inner and outer Peirce trivial, we say $e$ is Peirce trivial. \end{definition}     

For a ring $R$ with a unity, $e$ is inner (respectively, outer) Peirce trivial if and only if $eR(1-e)Re=\{0\}$ (respectively, $(1-e)ReR(1-e)=\{0\}$); moreover, $e$ is  inner Peirce trivial if and only $f=1-e$ is outer Peirce trivial. Let 

\medskip

\centerline{$\mathfrak{P}_{\rm it}(R), \mathfrak{P}_{\rm ot}(R)$ and $\mathfrak{P}_{\rm t}(R)$}

\medskip

\noindent denote the set of all inner Peirce trivial idempotents, all outer Peirce trivial idempotents and all Peirce trivial idempotents of $R$, respectively. Note that $\mathcal{B}(R) \subseteq \mathfrak{P}_t(R).$

\begin{example}
\label{two} Inner and outer Peirce trivialities are independent properties of idempotents. Let $R_1=\Z, R_2={\Z}/8\Z={\Z}_8, M_{12}={\Z}_4, M_{21}={\Z}_2$, together with tensor products $M_{12}\otimes_{R_2} M_{21}\cong {\Z}_2\mapsto 0\in R_1$ and $M_{21}\otimes_{R_1} M_{12}\cong {\Z}_2\cong 4R_2$, respectively. Then in $R = \left[ \begin{array}{cc}
R_1 & M_{12} \\
\\
M_{21} & R_2 \end{array}\right]$ the elements $e = \left[ \begin{array}{cc}
1 & 0 \\
\\
0 & 0 \end{array}\right]$ and $f = \left[ \begin{array}{cc}
0 & 0 \\
\\
0 & 1 \end{array}\right]$ are idempotents. Moreover, $e$ is inner Peirce trivial, but not outer Peirce trivial, and $f$ is outer Peirce trivial, but not inner Peirce trivial.
\end{example}

\begin{proposition}
Let $R = \left[ \begin{array}{cc} R_1 & M_{12} \\
\\
M_{21} & R_2 \end{array}\right] \in  \mathcal{T}_2$ and assume $\alpha = \left[ \begin{array}{cc} e & m \\
\\ n & f \end{array}\right] \in R$.
\begin{enumerate}
	\item [(1)] Then $\alpha = \alpha^2$ if and only if $e=e^2, f=f^2, em+mf =m$ and $ne  +fn =n$.
	\item[(2)] If $\alpha = \alpha^2$ and $e$ and $f$ are central idempotents, then $\alpha \in \mathfrak{P}_{\rm t}(R)$. In particular, if $R_1$ and $R_2$ are commutative, then $\mathfrak{P}_{\rm t}(R) = \{\alpha \in R \ \vert \ \alpha = \alpha^2\}$. 
\end{enumerate}
\end{proposition}

\begin{proof}
The proof is a straightforward calculation using Definition 1.1.
\end{proof}

Note that Proposition 1.3(2) is, in general, not true when $R \in \mathcal{T}_n$ for $n > 2$ (see Example 1.9).

As a consequence of Definition \ref{semitrivial}, one has the following descriptions:

\begin{lemma}
\label{innertrivial} For $e^2=e\in R$ the following claims are equivalent:
\begin{enumerate}
\item $e$ is inner Peirce trivial.
\item $e\ul{\ell}_R(e)=eR(1-e)$ is a right ideal of $R$.
\item $\ul{r}_R(e)e=(1-e)Re$ is a left ideal of $R$.
\item $efge=efege$ for all idempotents $f, g\in R$.
\item $h: R \rightarrow eRe$, defined by $h(x) = exe$, is a surjective ring homomorphism.
\item $eRtRe =0$ for all $t \in R$ such that $ete=0$.
\item $ReR \subseteq \ul{\ell}_R((1-e)Re)$.
\end{enumerate}
\end{lemma}

\begin{proof} We show implication $4\Rightarrow 1$; the remaining implications are routine.
For any $x, y\in R$ simple computation shows $f:=e-ex+exe=ef=f^2$ and $g:=e-ye+eye=ge=g^2$, whence one has by assumption $e+exye-exeye=efge=efege=e$, implying
$exye=exeye$. Therefore $e$ is inner Peirce trivial. 
\end{proof}

\begin{lemma}
\label{outertrivial} For $e^2=e\in R$ the following claims are equivalent:
\begin{enumerate}
\item $e$ is outer Peirce trivial.
\item $e\ul{\ell}_R(e)=eR(1-e)$ is a left ideal of $R$.
\item $\ul{r}_R(e)e=(1-e)Re$ is a right ideal of $R$.
\item $feg+efege=fege+efeg$ for all idempotents $f,g \in R$.
\end{enumerate}
\end{lemma}

\begin{proof} Again, we show the implication $4\Rightarrow 1$; the remaining implications are routine. For any $x,y\in R$ simple computation shows $f:=e+xe-exe=fe=f^2$ and $g:=e+ey-eye=eg=g^2$, whence one has by assumption $feg+efege=fg+e=fege+efeg=f+g$. Therefore
we have the equality $e+(e+xe-exe)(e+ey-eye)=e+xe-exe+e+ey-eye$, from which one can obtain, after simplification, that $e$ is outer Peirce trivial.
\end{proof}

\begin{corollary} For $e^2=e\in R$ the following claims are equivalent:
\begin{enumerate}
\item $e$ is Peirce trivial.
\item $e\ul{\ell}_R(e)=eR(1-e)$ is an ideal of $R$.
\item $\ul{r}_R(e)e=(1-e)Re$ is an ideal of $R$.
\item $e, 1-e\in \mathfrak{P}_{\rm it}(R)$.
\end{enumerate}
\end{corollary}

From the above results, one can see that if $R$ is semiprime, then $\mathfrak{P}_{\rm it}(R) = \mathfrak{P}_{\rm ot}(R) = \mathfrak{P}_{\rm t}(R) = \mathcal{B}(R)$.

\begin{lemma}
Let $e, f \in R$ such that $e=e^2$ and $f=f^2$. 
\begin{enumerate}
	\item $e\in \mathfrak{P}_{\rm it}(R)$ implies $efe = (efe)^2, \ (ef)^2 = (ef)^3$ and $(fe)^2 = (fe)^3$.
	\item $e \in \mathfrak{P}_{\rm t}(R)$ implies $fef = (fef)^2$.
	\item $e, f \in \mathfrak{P}_{\rm it}(R)$ implies $efe, fef \in \mathfrak{P}_{\rm it}(R)$.
	\item If $R$ is a generalized matrix ring and $[e_{ij}] \in \mathfrak{P}_{\rm it}(R)$ (resp.~$\mathfrak{P}_{\rm ot}(R), \mathfrak{P}_{\rm t}(R)$), then $e_{ii} \in \mathfrak{P}_{\rm it}(R_i)$ (resp.~$\mathfrak{P}_{\rm ot}(R_i), \mathfrak{P}_{\rm t}(R_i)$).
\end{enumerate}
\end{lemma}

\begin{proof}
This proof is routine.
\end{proof}

\begin{lemma}
Let $c,e \in R$ such that $c=c^2$ and $e=e^2$. 

\begin{enumerate}
	\item $e\in \mathfrak{P}_{\rm it}(R)$ if and only if $\mathfrak{P}_{\rm it}(eRe) = eRe\cap \mathfrak{P}_{\rm it}(R)$.
	\item $\mathfrak{P}_{\rm t}(R) \cap cRc \subseteq \mathfrak{P}_{\rm t}(cRc)$.
	\item Assume $I \trianglelefteq R$. Then $\mathfrak{P}_{\rm it}(I) = I\cap \mathfrak{P}_{\rm it}(R)$.
\end{enumerate}

\end{lemma}

\begin{proof} (1) Clearly, $eRe\cap \mathfrak{P}_{\rm it}(R)\subseteq \mathfrak{P}_{\rm it}(eRe)$. Assume $e \in \mathfrak{P}_{\rm it}(R), \ c \in \mathfrak{P}_{\rm it}(eRe)$ and $x,y \in R$. Then $cxyc = c(exye)c = c((exe)(eye))c= c(exe)c(eye)c = cxcyc$. Thus $eRe\cap \mathfrak{P}_{\rm it}(R) = \mathfrak{P}_{\rm it}(eRe)$. Conversely, assume $eRe \cap \mathfrak{P}_{\rm it}(R) = \mathfrak{P}_{\rm it}(eRe)$. Since $e \in \mathfrak{P}_{\rm it} (eRe)$, then $e \in \mathfrak{P}_{\rm it}(R)$.

	(2) The proof of this part is straightforward.
	
	(3) Clearly, $I \cap \mathfrak{P}_{\rm it}(R) \subseteq \mathfrak{P}_{\rm it}(I)$. Let $f \in \mathfrak{P}_{\rm it}(I)$ and $x,y \in R$. Then $fxyf = f((fx)(yf))f = f(fx)f(yf)f= fxfyf$. Therefore $\mathfrak{P}_{\rm it}(I) = I \cap \mathfrak{P}_{\rm it}(R)$.
 \end{proof}

\begin{example}
In general, for $c \in \mathfrak{P}_{\rm t}(R)$,  \ $\mathfrak{P}_{\rm t}(R) \cap cRc \subsetneq \mathfrak{P}_{\rm t}(cRc)$. Let $R$ be the $3$-by-$3$ upper triangular matrix ring over a ring $A$ with $e=E_{22} \in R$ and $c = E_{22} + E_{33}$. Then $c \in \mathfrak{P}_{\rm t}(R)$ and $e \in \mathfrak{P}_{\rm t}(cRc)$, but $e\not\in \mathfrak{P}_{\rm ot}(R)$. Thus, $\mathfrak{P}_{\rm t}(R) \cap cRc  \subsetneq \mathfrak{P}_{\rm t}(cRc)$.
\end{example}

In [BHKP] (also see [AvW1] and [AvW2]), it is shown that a ring $R$ has a generalized triangular matrix form if and only if it has a set of left (or right) triangulating idempotents. Such a set is an ordered complete set of orthogonal idempotents which are contructed from $\mathcal{S}_{\ell}(R)$ and $\mathcal{S}_r(R)$. 

Our next result and results from Sections 2 and 3 show that $\mathfrak{P}_{\rm it}(R)$ and $\mathfrak{P}_{\rm t}(R)$ can be used to naturally extend the notion of a generalized triangular matrix ring. Moreover, the inherent symmetry in the definitions of $\mathfrak{P}_{\rm it}(R)$ and $\mathfrak{P}_{\rm t}(R)$ frees us from the "ordered" condition on sets of idempotents when characterizing these natural extensions.

\begin{proposition} (1) $\mathcal{S}_l(R) \cup \mathcal{S}_r(R) \subseteq \mathfrak{P}_{\rm t}(R)$.

(2) Let $\{e_1, \ldots, e_n\}$ be a set of left or right triangulating idempotents of $R$. Then $\{e_1, \ldots, e_n\} \subseteq \mathfrak{P}_{\rm it}(R)$.
\end{proposition}

\begin{proof} (1) This part is immediate from the definitions.

	(2) From [BHKP, p.~560 and Corollary 1.6] and Lemma 1.8, each $e_i \in \mathfrak{P}_{\rm it}(R)$.
\end{proof}

From Proposition 1.10, $\mathcal{B}(R) \subseteq \mathcal{S}_l(R) \cup \mathcal{S}_r(R) \subseteq \mathfrak{P}_{\rm t}(R) \subseteq \mathfrak{P}_{\rm it}(R)$  ($\mathfrak{P}_{\rm ot}(R)$).  If $R$ is semiprime these containment relations become equalities by Lemmas 1.4 and~1.5.

\begin{example}
Let $A$ be a ring whose only idempotents are $0$ and $1$. Assume $0 \neq X, Y \trianglelefteq A$. Using Proposition 1.3 we obtain:

\begin{enumerate}
	\item Let $R = \left[ \begin{array}{ll} A & X \\
	0 & A \end{array}\right]$. Then $\mathcal{S}_l (R) \cup \mathcal{S}_r(R) = \mathfrak{P}_{\rm t}(R) = \{e \ \vert \  e = e^2 \in R\}$.
	\item Let $R = \left[ \begin{array}{ll} A & X \\
	Y & A \end{array}\right]$ and $XY = 0 =YX$. Then $\mathcal{S}_r(R) = \mathcal{S}_l(R) = \{0,1\} \subsetneq \mathfrak{P}_{\rm t}(R) = \{e \ \vert \ e = e^2 \in R\}$.
	\item Let $B$ be a subring of $A$ with $X \trianglelefteq B$ and $R = \left[ \begin{array}{cc} A & X \\
	A/X & B \end{array}\right]$. Then $\mathcal{S}_r(R) = \mathcal{S}_l(R) = \{0,1\} \subsetneq \mathfrak{P}_{\rm t}(R) = \{e \ \vert \ e = e^2 \in R\}$.
\end{enumerate}
\end{example}

\smallskip
We conclude Section 1 by showing in the next results (1.12 - 1.16) that given a base ring~$R$, an overring $T$, and a set $\mathcal{E}$ contained in $\mathfrak{P}_{\rm it}(T) \cup \mathfrak{P}_{\rm ot}(T)$, there is a significant transfer of information between $R$ and~$S$, where $S$ is the subring of~$T$ generated by $R$ and $\mathcal{E}$.  These results indicate the importance of the inner and outer Peirce trivial idempotents.

Let $S$ be an overring of $R$.  We consider the following properties between prime ideals of $R$ and $S$ (see 
    [BPR2, pp.~295-296] or [R1, p.~ 292]).
\medskip		
			
    (1)  {\it Lying over (LO)}.  For any prime ideal $P$ of $R$, there exists a prime ideal $Q$ of $S$ such that $P = Q \cap R$.
		
    (2)  {\it Going up (GO)}.  Given prime ideals $P_1 \subseteq P_2$ of $R$ and $Q_1$ of $S$ with $P_1 = Q_1 \cap R$, there exists a prime ideal $Q_2$ of $S$ with 
         $Q_1 \subseteq Q_2$ and $P_2 = Q_2 \cap R$.
					
    (3)  {\it Incomparable (INC)}. Two different prime ideals of $S$ with the same contraction in $R$ are not comparable.

\begin{lemma}
Let $T$ be a ring, $R$ a subring of $T$, 
$$\mathcal{E}_{\mathcal{P}}  = \{e = e^2 \in T \ \vert \  e + \mathcal{P}(T) \ \mbox{is central in} \ T/\mathcal{P}(T)\},$$
and $S = \langle R \cup \mathcal{E} \rangle_T$, where $\emptyset \neq \mathcal{E} \subseteq \mathcal{E}_\mathcal{P}$ Then:

\begin{enumerate}
	\item $\mathfrak{P}_{\rm it}(T) \cup \mathfrak{P}_{\rm ot}(T) \subseteq \mathcal{E}_\mathcal{P}$.
	\item If $K$ is a prime ideal of $S$, then $R/(K\cap R) \cong S/K$.
	\item LO, GU and INC hold between $R$ and $S$.
\end{enumerate}
\end{lemma}

\begin{proof} (1) This part follows from Lemmas 1.4 and 1.5.

 (2) and (3). The proof of these parts is similar to that in [BPR1, Lemma 2.1] or [BPR2, Lemma 8.3.26].
\end{proof}

Recall that a ring $R$ is {\it strongly $\pi$-regular} if for each $x$ there is a positive integer~$n$ (depending on $x$) such that $x^n \in x^{n+1} R$.

\begin{theorem}
Let $\mathcal{C}$ be a property of rings such that a ring $A$ has property $\mathcal{C}$ if and only if every prime factor of $A$ has property $\mathcal{C}$. Assume $T$ is a ring, $R$ is a subring of $T$ and $S:= \langle R \cup \mathcal{E} \rangle_T$, where $\emptyset \neq \mathcal{E} \subseteq \mathcal{E}_{\mathcal{P}},$ with $\mathcal{E}_{\mathcal{P}}$ as in Lemma 1.12. Then $R$ has property $\mathcal{C}$ if and only if $S$ has property $\mathcal{C}$. In particular, $R$ is strongly $\pi$-regular if and only if $S$ is strongly $\pi$-regular. 
\end{theorem}

\begin{proof}
Assume $R$ has property $\mathcal{C}$ and $K$ is a prime ideal of $S$. From Lemma~1.12(2), $R/(K \cap R)$ is a prime ring, and so 
$R/(K \cap R)$ has property $\mathcal{C}$. Hence $S/K$ has property $\mathcal{C}$. Therefore $S$ has property $\mathcal{C}$.

Conversely, assume $S$ has property $\mathcal{C}$ and $P$ is a prime ideal of $R$. From Lemma~1.12(3), LO holds between $R$ and $S$. So there exists a prime ideal $Q$ of $S$ such that $Q \cap R = P$. By Lemma 1.12(2), $R/P = R/(Q \cap R) \cong S/Q$. Hence $R/P$ has  property $\mathcal{C}$. Therefore $R$ has property $\mathcal{C}$. 

From [FS], a ring $A$ is strongly $\pi$-regular if and only if every prime factor of $A$ is a strongly $\pi$-regular ring. 
\end{proof}

See [GW] for the definition of a special radical. Observe that the prime, Jacobson, and nil radicals are included in the collection of special radicals.

\begin{theorem}
Let $R$ be a subring of a ring $T,  \ \emptyset \neq \mathcal{E} \subseteq \mathcal{E}_\mathcal{P}$, and $S = \langle R \cup \mathcal{E} \rangle_T$. Then:
\begin{enumerate}
	\item $\rho(R) = \rho(S) \cap R$, where $\rho$ is any special radical.
	\item The classical Krull dimensions of both $S$ and $R$ are equal.
	\item If $S$ is a von Neumann regular ring, then so is $R$.
\end{enumerate}
\end{theorem}

\begin{proof}
Using Lemma 1.12, the proof is similar to [BPR1, Theorem 2.2] or [BPR2, Theorem 8.3.28].
\end{proof}

\begin{lemma}
Let $T \in \mathcal{T}_n \ (n > 1)$ and
$$\mathcal{E}_k = \{[t_{ij}]\in T | \, t_{kj} \in M_{kj} \ \mbox{for} \ j \neq k, t_{kk} = 1 \in T_k \ \mbox{and all other entries are zero} \}.$$
Then $\cup^n_{k=1} \mathcal{E}_k \subseteq \mathfrak{P}_{\rm it}(T)$.
\end{lemma}

\begin{proof}
The proof of this result is a routine but tedious application of Definition~1.1.
\end{proof}

\begin{theorem}
 Let $T \in \mathcal{T}_n \ (n > 1)$, $R = \mathcal{D}(T)$, and $\mathcal{E} = \cup^n_{k=1} \mathcal{E}_k$ (as in Lemma~1.15). Then:
\begin{enumerate}
	\item $\langle R \cup \mathcal{E}\rangle_T = S = T$.
	\item $R$ has property $\mathcal{C}$ (as in Theorem 1.13) if and only if $T$ has property $\mathcal{C}$.
  \item  $\rho(R) = \rho(T) \cap R$.
  \item  The classical Krull dimension of both $R$ and $T$ are equal.
\end{enumerate}
\end{theorem}

\begin{proof}
Use Lemmas 1.12 and 1.15 and Theorems 1.13 and 1.14.
\end{proof}

This result extends [TLZ, Corollary 3.6].

\vskip0.5cm

\section{Characterization of $\mathcal{T}_n$}

\begin{lemma}
Let $\{e_1, \ldots, e_n\}$ be a complete set of orthogonal idempotents of $R$.
\begin{enumerate}
	\item $e_i \in \mathfrak{P}_{\rm it}(R)$ if and only if $e_iRe_jRe_i =0$ for all $j \neq i$.
	\item $e_j \in \mathfrak{P}_{\rm ot}(R)$ if and only if $e_iRe_jRe_k =0$ for all $i \neq j$ and $k \neq j$.
	\item $\{e_1, \ldots, e_n\} \subseteq \mathfrak{P}_{\rm ot}(R)$ if and only if $\{e_1, \ldots, e_n\} \subseteq \mathfrak{P}_{\rm t}(R)$.
\end{enumerate}
\end{lemma}

\begin{proof} (3) follows obviously from (1) and (2). Furthermore, (1) and (2) are immediate consequences of the definition of Peirce trivial idempotents by observing
$$0=e_iR(1-e_i)Re_i=e_iR(\sum_{j\neq i}e_j)Re_i=\sum_{j\neq i} e_iRe_jRe_i\Leftrightarrow \forall j\neq i\,\, e_iRe_jRe_i=0,$$
and
$$0=(1-e_j)Re_jR(1-e_j)=\oplus_{i\neq j\neq k} e_iRe_jRe_k \Leftrightarrow \forall i\neq j\neq k\,\, e_iRe_jRe_k=0.$$ 
\end{proof}

Lemma 2.1 shows remarkably that inner and outer Peirce trivial idempotents behave quite differently when they are considered together as a complete set of idempotents although their definition seems very symmetric! Lemma 2.1 shows clearly the equivalence of the first three statements in the next result. 

\begin{theorem}
Let $\{e_1, \ldots, e_n\}$ be a complete set of orthogonal idempotent elements of $R$. The following conditions are equivalent:
\begin{enumerate}
	\item $R^\pi \in \mathcal{T}_n$.
	\item $e_iRe_jRe_i = 0$, for all $i \neq j$.
	\item $\{e_1, \ldots, e_n\} \subseteq \mathfrak{P}_{\rm it}(R)$.
	\item $\mathcal{D}(R^\pi)^-$ is a right ideal of $R^\pi$.
\end{enumerate}
\end{theorem}

\begin{proof} (3) $\Leftrightarrow$ (4) Let $X = \mathcal{D}(R^\pi)^-$. Observe that the $(i,i)$-position of $XR^\pi$ is $\sum_{k\neq i} e_i R_kRe_i = e_i R(1-e_i) Re_i$; and the $(i,j)$-position of $XR^\pi$ is $\sum_{k\neq i, \ k\neq j} e_i Re_kRe_j \subseteq e_iRe_j$.  Therefore $XR^\pi \subseteq X$ if and only if each $e_i \in \mathfrak{P}_{\rm it}(R)$. 
\end{proof}

Observe that from Lemma 1.4 and Theorem 2.2, any property that is preserved by a surjective ring homomorphism passes from a ring in $\mathcal{T}_n$ to its diagonal rings.

\begin{corollary}
 Let $R$ be an $n$-by-$n$ generalized matrix ring. Then the following conditions are equivalent:
 
\begin{enumerate}
	\item $R \in \mathcal{T}_n$.
	\item Let $[a_{ij}], [b_{ij}] \in R$ with $[c_{ij}] = [a_{ij}] [b_{ij}]$. Then $c_{ii} = a_{ii}b_{ii}$ for all $i$ and $j$.
	\item $\{E_{ii} \in R \ \vert \ i = 1, \ldots, n\} \subseteq \mathfrak{P}_{\rm it}(R)$.
\end{enumerate}
\end{corollary}

Thus $\mathcal{T}_n$ is exactly the class of $n$-by-$n$ generalized matrix rings in which the diagonal entries of the product 
    of two matrices is completely determined by the corresponding entries of the diagonals of the factor matrices.
		
Note that the idempotents in a generalized matrix ring are not characterized. However, for $R\in \mathcal{T}_n$, $e=e^2\in R$ if and only if $e=[e_{ij}]$, where $e_{ii}=e_{ii}^2$ and $e_{ij}=\sum_{k=1}^n e_{ik}e_{kj}$ for $i\ne j$.

\begin{proposition}
Assume $\{e_1, \ldots, e_n\}$ is a complete set of orthogonal idempotents of $R$, and $X = \mathcal{D}(R^\pi)^-$.
\begin{enumerate}
	\item If $\{e_1, \ldots, e_n\} \subseteq \mathfrak{P}_{\rm it}(R)$, then $X^n =0$ and $\oplus_{i=1}^nR_i$ is a homomorphic image of $R^\pi$ with kernel $X$, where $R_i=e_iRe_i$.
	\item If $\{e_1, \ldots, e_n\} \subseteq \mathfrak{P}_{\rm t}(R)$, then $X^2=0$.
\end{enumerate}
\end{proposition}

\begin{proof} (1) Observe that the $(i,j)$-position of $X^{n-1}$ is a sum of terms where each term is an element of $e_i Re_{\alpha_1}Re_{\alpha_2}R \cdots e_{\alpha_{n-2}}Re_j$ where $\alpha_k \in \{1, \ldots, n\} - \{i, j\}$. Since $\{e_1, \ldots, e_n\} \subseteq \mathfrak{P}_{\rm it}(R)$, it follows that $X^n =0$. The second part follows from Lemma 1.4(5) and Theorem~2.2.
	
	(2) This part follows from Lemma 2.1(2).
\end{proof}

\begin{example}
Let $A$ be a ring and $X, Y \trianglelefteq A$ such that $X^2 \subseteq Y$. Take
$$R = \left[ \begin{array}{ccc} A & X & Y \\
X & A & X\\
Y & X & A \end{array}\right].$$
Then routine calculation yields:
\begin{enumerate}
	\item $E_{22} \in \mathfrak{P}_{\rm t}(R)$ if and only if $X^2=0$.
	\item $\{E_{11}, E_{22}, E_{33} \} \subseteq \mathfrak{P}_{\rm it}(R)$ if and only if $X^2 =0 =Y^2$.
	\item $\{E_{11}, E_{22}, E_{33} \} \subseteq \mathfrak{P}_{\rm t}(R)$ if and only if $X^2 =0 =Y^2$ and $XY =0=YX$.
\end{enumerate}

For an illustration of (2) and (3), let $B$ be a ring and $A = B[x,y]/(x^2, y^2)$ and $A = B[x,y]/(x^2, y^2, xy)$, respectively.
\end{example}

\smallskip

The next three results (2.6 - 2.8) indicate a transfer of important ring properties between~$\mathcal{D}(R)$ and $R\in\mathcal{T}_n$.  

\begin{corollary}
Let $R \in \mathcal{T}_n \ (n > 1)$. Then $\mathcal{D}(R)$ satisfies each of the following conditions if and only if $R$ does so:

\begin{enumerate}
	\item semilocal,
	\item semiperfect,
	\item left (or right) perfect,
	\item semiprimary,
	\item bounded index (of nilpotence).
\end{enumerate}
\end{corollary}

\begin{proof}
Observe that if $R$ satisfies any of (1) - (5), then so does $eRe$ for each $e=e^2 \in R$. Hence, if $R$ satisfies any of (1) - (5), then so does $\mathcal{D}(R)$. 

Conversely, first assume that $\mathcal{D}(R)$ is semilocal. From Proposition 2.4(1) we have that $\mathcal{D}(R)^- \trianglelefteq R,  \ R/(\mathcal{D}(R)^-) \cong \mathcal{D}(R)$ and $\mathcal{D}(R)^- \subseteq \mathcal{J}(R)$. Therefore, $R/\mathcal{J}(R) \cong (R/(\mathcal{D}(R)^-))/(\mathcal{J}(R)/(\mathcal{D}(R)^-)) \cong \mathcal{D}(R)/\mathcal{J}(\mathcal{D}(R))$ is semisimple artinian. Thus $R$ is semilocal. Parts (2) - (4) are proved similarly. Now assume $\mathcal{D}(R)$ has bounded index $k$, and let $v$ be a nilpotent element of $R$. Then $v = d + x$ where $d \in \mathcal{D}(R)$ and $x \in \mathcal{D}(R)^-$ and $v^m =0$. Then $0 = v^m = d^m+y$ where $y \in \mathcal{D}(R)^-$. So $d^m =-y \in \mathcal{D}(R)^-$. Hence $d$ is nilpotent, so $d^k =0$. Then $v^k = d^k +w = w \in \mathcal{D}(R)^-$. Hence $v^{kn} =0$. Thus $R$ has bounded index less than or equal to $kn$.
\end{proof}

Corollary 2.6(3) extends [TLZ, Corollary 3.8]. In [ABP], the authors determine several generalizations of the condition that a ring satisfies a polynomial identity. With these generalizations they were able to extend classical theorems by Armendariz and Steenberg, Fisher, Kaplansky, Martindale, Posner and Rowen. Two of these generalizations are: (1) a ring $R$ is an {\it almost PI-ring} if every prime factor ring of $R$ is a PI-ring; (2) $R$ is an {\it instrinsically PI-ring} if every nonzero ideal contains a nonzero PI-ideal of $R$.

\begin{corollary}
Let $R \in \mathcal{T}_n \ (n> 1)$. Then:

\begin{enumerate}
	\item $\mathcal{D}(R)$ satisfies a PI if and only if $R$ does so.
	\item If $\mathcal{D}(R)$ is commutative, then $R$ satisfies $(xy - yx)^n =0$ for all $x, y \in R$.
	\item $\mathcal{D}(R)$ is almost PI if and only if $R$ is almost PI.
	\item If $\mathcal{D}(R)$ is intrinsically PI, then $R$ is instrinsically PI.
\end{enumerate}
\end{corollary}

\begin{proof} (1) Since subrings of a PI-ring are PI-rings, if $R$ is a PI-ring then so is~$\mathcal{D}(R)$. Conversely, assume $\mathcal{D}(R)$ is a PI-ring which satisfies the PI  $p$. Then $R$ satisfies~$p^n$.
	
	(2) This part follows from (1).
	
	(3) and (4). By Proposition 2.4 we have $\mathcal{D}(R)^- \trianglelefteq R, \ R/(\mathcal{D}(R)^-) \cong \mathcal{D}(R)$ and $(\mathcal{D}(R)^-)^n =0$. Now (3) follows from [ABP, Theorem 1.6(i)], and (4) follows from [ABP, Theorem 1.6(ii)].
\end{proof}

Let $R$ be an $n$-by-$n$ generalized matrix ring, and let 

\medskip

\centerline{UT$(R)$ and  LT$(R)$}

\medskip

\noindent be the $n$-by-$n$ upper and lower generalized triangular matrix rings, respectively, formed from $R$. Our next result shows that elements of $\mathcal{T}_n$ are subdirect products of generalized triangular matrix rings.

\begin{proposition}
Let $R \in \mathcal{T}_n \ (n> 1)$. Then there is a ring monomorphism \hfill\break $\psi: R \rightarrow {\rm UT}(R) \times {\rm LT}(R)$ such that $R$ is a subdirect product of UT$(R)$ and~LT$(R)$.
\end{proposition}

\begin{proof}
Let $[m_{ij}] \in R$. Define $\psi ([m_{ij}]) = ([a_{ij}], [b_{ij}])$ where
$$a_{ij} = \left\{ \begin{array}{ll} m_{ij}, & \mbox{for} \ j \geq i\\
0, & \mbox{elsewhere}, \end{array}\right.  \quad \mbox{and}$$

$$b_{ij} = \left\{ \begin{array}{ll} m_{ij}, & \mbox{for} \ i \geq j \\
0, & \mbox{elsewhere} . \end{array}\right.$$
A routine argument yields that $\psi$ is a ring monomorphism and that $R$ is a subdirect product of UT$(R)$ and LT$(R)$.
\end{proof}

\begin{definition}
Let $R$ be an $n$-by-$n$ generalized matrix ring. Let $R^{\rm la}$ denote the {\it lower annihilating} subring 

$$\left[ \begin{array}{cccc}
R_1 & M_{12} & \cdots & M_{1n}\\
\\
\ul{r}_{M_{21}}(M_{12}) \cap \ul{\ell}_{M_{21}}(M_{12}) & R_{2} & \ddots & \vdots \\

\\
\vdots & \ddots & \ddots & M_{n-1,n} \\
\\
\ul{r}_{M_{n1}}(M_{1n}) \cap \ul{\ell}_{M_{n1}}(M_{1n})  & \cdots & \ul{r}_{M_{n,n-1}}(M_{n-1,n}) \cap \ul{\ell}_{M_{n,n-1}}(M_{n-1,n}) & R_n \end{array}\right]$$

\smallskip

\noindent of $R$, and let $R^{\rm ua}$ denote the {\it upper annihilating} subring 

$$\left[ \begin{array}{cccc}
R_1 & \ul{r}_{M_{12}}(M_{21}) \cap \ul{\ell}_{M_{12}}(M_{21})  &  \cdots &  \ul{r}_{M_{1n}}(M_{n1}) \cap \ul{\ell}_{M_{1n}}(M_{n1})         \\
\\
 M_{21} & R_{2} & \ddots & \vdots \\
\\
\vdots & \ddots & \ddots &  \ul{r}_{M_{n-1,n}}(M_{n,n-1}) \cap \ul{\ell}_{M_{n-1,n}}(M_{n,n-1})      \\
\\
  M_{n1}   & \cdots & M_{n,n-1} & R_n \end{array}\right]$$
	
\smallskip	

\noindent of $R$.
\end{definition}

Note that $R^{\rm la}$ and $R^{\rm ua}$ are subrings of both $R$ and $\ol{R}$. Moreover, if $R$ is the $n$-by-$n$ matrix ring over  a ring $A$, then $R^{\rm la}$ and $R^{\rm ua}$ are the $n$-by-$n$ upper and lower triangular matrix rings over $A$, respectively.

\begin{example}
Let $A$ and $B$ be rings. Let $R = \left[ \begin{array}{cc} A \times B & A \times \{0\} \\
A \times B & A \times B \end{array} \right]$. Then $R^{\rm la} = \left[ \begin{array}{cc} A \times B & A \times \{0\} \\
\{0\} \times B & A \times B \end{array}\right]$, and $R^{\rm ua} = \left[ \begin{array}{cc} A \times B & \{0\} \\
A \times B & A \times B \end{array}\right]$.
\end{example}

\begin{lemma}
Let $R$ be an $n$-by-$n$ generalized matrix ring. Then
${\rm Cen}(R) = {\rm Cen}(\ol{R}) = \{[c_{ij}] \in R \ \vert \ c_{ij} =0 \ \mbox{if} \ i \neq j, \ c_{ii} \in {\rm Cen}(R_i) \ \mbox{for all} \ i, \ \mbox{and} \ c_{ii}m_{ij} = m_{ij}c_{jj} \ \mbox{for all} \ m_{ij} \in M_{ij} \ \mbox{if} \ i \neq j\}.$
\end{lemma}

\begin{proof}
First show the result holds when $n=2$. For the $n$-by-$n$ case, block the matrix ring into a $2$-by-$2$ generalized matrix ring and use induction.
\end{proof}

Let $S$ be a subring of a ring $R$. It is well known (see [W, p.~26]) that the $(S,S)$-bimodule structure of $S$ and $R$ is equivalent to the right $T$-module structure of~$S$ and $R$, respectively, where $T = S^{\rm op} \otimes_{\Z} S$, with $S^{\rm op}$ denoting the opposite ring of~$S$.

The next three results (2.12 - 2.14) indicate the transfer of significant information between an $n$-by-$n$ generalized matrix ring and certain subrings which are maximal with respect to being in $\mathcal{T}_n$.

In particular, Theorem 2.12 shows that for any ring $R$ with a complete set of orthogonal idempotents $\{e_i\}^n_{i=1} \  (n> 1)$ there are subrings $S$ containing $\{e_i\}^n_{i=1}$ which are maximal with respect to $S^\pi$ being in~$\mathcal{T}_n$ and $S_T$ is right essential in~$R_T$, where $T = S^{\rm op} \otimes_{\Z}S$. Moreover, this result and its consequences provide a connection between the structure of an arbitrary generalized matrix ring and the structure of rings in $\mathcal{T}_n$ (see Question B in the introduction).

\begin{theorem}
Let $R$ be an $n$-by-$n$ generalized matrix ring, and $S$ denotes either~$R^{\rm la}$ or $R^{\rm ua}$.

\begin{enumerate}
	\item $S$ is a subring of $R$ maximal with respect to being in $\mathcal{T}_n$.
	\item Let $0 \neq y \in R$. Then either $0 \neq syE_{jj} \in S$ or $0 \neq E_{ii} y t \in S$ for some $s, t, E_{ii}, E_{jj} \in S$.
	\item Every nonzero $(S,S)$-bisubmodule of $R$ has nonzero intersection with $S$. Thus every nonzero ideal of $R$ has nonzero intersection with $S$, and $S_T$ is right essential in $R_T$ where $T = S^{\rm op} \otimes_\Z S$.
	\item ${\rm Cen}(R) = {\rm Cen}(R^{\rm la}) \cap {\rm Cen}(R^{\rm ua})     \subseteq {\rm Cen}(\mathcal{D}(R))$.
	\item $\begin{array}{lll} {\rm U}(\ol{R}) & = & \{u+x \ \vert \  \ u \in {\rm U}(\mathcal{D}(R)) \ \mbox{and} \ x \in \mathcal{D}(\ol{R})^- \}, \ \mbox{and}\\
	{\rm U}(S) & = & \{u +y \ \vert \ u \in {\rm U}(\mathcal{D}(R)) \ \mbox{and} \ y \in \mathcal{D}(S)^-\} \subseteq {\rm U}(R).\end{array}$
\end{enumerate}
\end{theorem}

\begin{proof} (1) Suppose that there is a subring $Y$ of $R$ such that $S$ is properly contained in $Y$. Assume $S = R^{\rm la}$. Then there exists $y \in Y$ with an entry $y_{ij}$ for some $i,j$ with $i > j$ such that $y_{ij} \not\in \ul{r}_{M_{ij}}(M_{ji}) \cap \ul{\ell}_{M_{ij}}(M_{ji})$. Then either $y_{ij} \not\in \ul{\ell}_{M_{ij}}(M_{ji})$, or $y_{ij} \in \ul{\ell}_{M_{ij}}(M_{ji})$ but $y_{ij} \not\in \ul{r}_{M_{ij}}(M_{ji})$. 

If $y_{ij} \not\in \ul{\ell}_{M_{ij}}(M_{ji})$, there exists $k \in M_{ji}$ such that $y_{ij} k \neq 0$. Let $t$ be the $n$-by-$n$ matrix with $k$ in the $(j,i)$-position and zero elsewhere. Then $t \in S \subseteq Y$, and so $0 \neq yt \in Y$. However, since $0 \neq y_{ij}k\in M_{ij} M_{ji} \in R_i$, we have that $Y \not\in \mathcal{T}_n$.
	
	If $y_{ij} \in \ul{\ell}_{M_{ij}}(M_{ji})$ but $y_{ij} \not\in \ul{r}_{M_{ij}}(M_{ji})$, then there exsits $h \in M_{ji}$ such that $hy_{ij} \neq 0$. Let $s$ be the $n$-by-$n$ matrix with $h$ in the $(j,i)$-position and zero elsewhere. Then $s \in S \subseteq Y$, and so $0 \neq sy \in Y$. However, since $0 \neq hy_{ij} \in M_{ji}M_{ij}\in R_j$, it follows that $Y \not\in \mathcal{T}_n$. 
	
	Therefore $S$ is a subring of $R$  maximal with respect to being in $\mathcal{T}_n$. The argument when $S = R^{\rm ua}$ is similar.
	
	(2) Again let $S = R^{\rm la}$ and $0 \neq y \in R$. If $y \in S$, we are finished. So assume $y \not\in S$. Then as in part (1) there exists an entry $y_{ij}$ of $y$ for some $i, j$ with $i > j$ such that $y_{ij} \not\in \ul{r}_{M_{ij}} (M_{ji}) \cap \ul{\ell}_{M_{ij}}(M_{ji})$. As in part (1), we obtain $s, t \in S$. Then $0 \neq sy E_{jj} \in S$ or $0 \neq E_{ii} yt\in S$. The argument when $S = R^{\rm ua}$ is similar.
	
	(3) This part is a consequence of (2).
	
	(4) This part follows from Lemma 2.11.
	
	(5) By Proposition 2.4(1), $\mathcal{D}(\ol{R})^- \trianglelefteq \ol{R}$ such that $(\mathcal{D}(\ol{R})^-)^n =0$. Let $u \in \mathcal{D}(R)$ and $x \in \mathcal{D}(\ol{R})^-$. Then $(u+x)(u^{-1}+x) = 1+ ux+xu^{-1}+x^2$. But $ux+xu^{-1}+x^2 \in \mathcal{D}(\ol{R})$, hence $(u+x)(u^{-1}+x) =w$ where $w \in {\rm U}(\ol{R})$. Therefore $u+x \in {\rm U}(\ol{R})$.
	
	Now assume $v \in {\rm U}(\ol{R})$. Then there exist $d \in \mathcal{D}(\ol{R})$ and $y \in \mathcal{D}(\ol{R})^-$ such that $v = d+y$. So $1 = dv^{-1}+yv^{-1}$. Hence $dv^{-1} = 1-yv^{-1}$. Since $yv^{-1} \in \mathcal{D}(\ol{R})^-,  \ dv^{-1} \in {\rm U}(\ol{R})$. So $d \in {\rm U}(\ol{R})$. Therefore ${\rm U}(\ol{R}) = \{u+x \ \vert \ u \in {\rm U}(\mathcal{D}(R))$ and $x \in \mathcal{D}(\ol{R})^-\}$. 
	
	The remainder of the proof is due to the above argument and the fact that $S$ is a subring of $R$ and $\ol{R}$.  
\end{proof}

Note that if $n=2$, then in Theorem 2.12(2), there is no need for the $E_{ii}$ and~$E_{jj}$.
\vskip0.2truecm

\noindent QUESTION: When is ${\rm U}(R)$ generated by ${\rm U}(R^{\rm la}) \cup {\rm U}(R^{\rm ua})$?

\begin{corollary}
Let $R$ be an $n$-by-$n$ generalized matrix ring, $S$ denotes $R^{\rm la}$ or~$R^{\rm ua}$ and $T = S^{\rm op} \otimes_\Z S$. Then:

\begin{enumerate}
	\item $S$ is maximal among subrings $Y$ of $R$ for which $\{E_{ii} \}^n_{i=1} \subseteq \mathfrak{P}_{\rm it}(R)$.
	\item The sum of the minimal ideals of $S$ equals ${\rm Soc}(S_T) = {\rm Soc}(R_T)$.
	\item The uniform dimension of $_SS_S$ equals the uniform dimension of $_SR_S$ equals the uniform dimension of $S_T$ equals the uniform dimension of $R_T$.
\end{enumerate}
\end{corollary}

\begin{proof} (1) This part is a consequence of Theorems 2.2 and 2.12.
	
	(2) and (3). These parts are consequences of Theorem 2.12(3).
\end{proof}

Our next result demonstrates that useful information can be transferred from the diagonal rings $R_i$ of a generalized matrix ring $R$ to $R$ itself via Theorems 1.16 and~2.12. Recall from [R2] and [CR] that an $n$-by-$n$ $(n> 1)$ matrix ring over a strongly $\pi$-regular ring is not, in general, a strongly $\pi$-regular ring.

\begin{corollary}
Let $R$ be an $n$-by-$n$ generalized matrix ring, and $S = R^{\rm la}$ or~$R^{\rm ua}$. If $\mathcal{D}(R)$ is strongly $\pi$-regular, then for each $0 \neq y \in R$ either:

\begin{enumerate}
	\item $y \in S$, in which case $y^n \in y^{n+1}S \subseteq y^{n+1}R$ for some positive integer $n$; or
	\item $y \not\in S$, in which case either $0 \neq syE_{jj} \in S$ and $(syE_{jj})^m \in (syE_{jj})^{m+1}S \subseteq (syE_{jj})^{m+1}R$, or $0 \neq E_{ii}yv \in S$ and $(E_{ii}yv)^k \in (E_{ii}yv)^{k+1}S \subseteq(E_{ii} yv)^{k+1}R$ for some $s,v,E_{ii}, E_{jj} \in S$ and positive integers $k, m, n$.
\end{enumerate}
\end{corollary} 

\begin{proof}
The proof follows from Theorems 1.16 and 2.12(2).
\end{proof}

Thus if $\mathcal{D}(R)$ is strongly $\pi$-regular, then $R$ is ``almost'' strongly $\pi$-regular.

\vskip0.4truecm

Next we introduce the notion of an ideal extending ring. The ideal extending condition is shown to be a Morita invariant.  Moreover, it is shown that important classes of rings have this property.  For example, the semiprime quasi-Baer rings are ideal extending, so this insures that every semiprime ring has a hull which is ideal extending (see Proposition 2.16).  The class of semiprime quasi-Baer rings includes the local multiplier $C^\ast$-algebras which means that every $C^\ast$-algebra can be embedded into its local multiplier $C^\ast$-algebra which is an ideal extending ring.  

As applications of the results in 2.12 - 2.14 we show in 2.16 - 2.21 that the ideal extending property transfers from a ring $A$ to a certain type of overring of $A$ which is in $\mathcal{T}_n$. 

Let $X$ and $Y$ be both left or both right ideals of a ring $R$ with $X \subseteq Y$. Then $X$ is {\it ideal essential} in $Y$ if for each $0 \neq I \trianglelefteq R$ such that $I \subseteq Y$, then $0 \neq X \cap I$. Note that if $R$ is a semiprime ring and $X, Y \trianglelefteq R$ with $X \subseteq Y$, then $X$ is ideal essential in $Y$ if and only if $X$ is right or left essential in $Y$.

\begin{definition}
We say $R$ is {\it ideal extending} if for each $X \trianglelefteq R$ there is an $e \in \mathcal{B}(R)$ such that $X$ is ideal essential in $eR$. 
\end{definition}

Note that every nonzero ideal of~$R$ is ideal essential in $R$ if and only if $\mathcal{B}(R) = \{0,1\}$ and $R$ is ideal extending. Some immediate examples of ideal extending rings are: $R$ is a prime ring, $R$ is an Abelian (i.e., every idempotent is central) right extending ring (e.g., $R$ is a direct sum of commutative uniform rings (see [DHSW])), or $R = \left[ \begin{array}{cc} A & M \\
0 & A \end{array}\right]$ where $A$ is a prime ring and $M \trianglelefteq A$.

The next result shows that the class of ideal extending rings is quite extensive. See [BPR1] or [BPR2] for undefined terminology.

\begin{proposition}
Assume $R$ is a semiprime ring. Then:
\begin{enumerate}
	\item [(1)] $R$ is ideal extending if and only if $R$ is quasi-Baer if and only if $R_R$ is FI-extending.
	\item[(2)] $R$ has an ideal extending hull.
\end{enumerate}
\end{proposition}

\begin{proof} (1) See [BMR, Theorem 4.7] or [BPR2, Theorem 3.2.37].

(2) This part follows from (1) and [BPR1, Theorem 3.3] or [BPR2, Theorem~8.3.17].
\end{proof}

From Proposition 2.16 it follows that every von Neumann algebra and every local multiplier algebra of a $C^\ast$-algebra are ideal extending as rings (see [K] and [BPR1, pp.~345-347] or [BPR2, pp.~380-407]).

\begin{theorem} (1) Let $\{R_i | i \in I\}$ be a set of rings. Then $\Pi_{i\in I} R_i$ is ideal extending if and only if each $R_i$ is so.

	(2) The ideal extending property is a Morita invariant.
\end{theorem}

\begin{proof} (1) The proof of this part is routine.

(2) Assume that $R$ is ideal extending. Then a straightforward argument shows that $R$ is ideal extending if and only if the ring of $n$-by-$n$ matrices over $R$ is ideal extending. Let $e$ be a full idempotent of $R$ (i.e. $ReR= R$) and $0\neq K \trianglelefteq e Re$. Then $RKR$ is ideal essential in $cR$ for some $c \in \mathcal{B}(R)$. Hence $K \subseteq ec(eRe)$, where $ec \in \mathcal{B}(eRe)$. Let $0 \neq X \trianglelefteq eRe$ such that $X \subseteq ec(eRe)$. Then $RXR \subseteq cR$, so $0 \neq Y = RXR \cap RKR$. If $eYe \neq 0$, there exists $y \in Y$ such that $0 \neq eye= \Sigma r_\alpha x_\alpha s_\alpha = \Sigma t_\beta k_\beta v_\beta$, where $r_\alpha, s_\alpha, t_\beta, v_\beta \in R$, $x_\alpha \in X$ and $k_\beta \in K$. So $x_\alpha = e x_\alpha e$ and $k_\beta = ek_\beta e$. Hence $0 \neq eye = \Sigma er_\alpha (ex_\alpha e)s_\alpha e = \Sigma et_\beta (ek_\beta e)v_\beta e \in X \cap K$.

Now assume $eYe = 0$. Since $e$ is full, $1 = \Sigma a_j e b_j$ for some $a_j,b_j\in R$. Let $0\ne w\in Y.$ Then $w = 1w1 = (\Sigma a_j e b_j) w (\Sigma a_j e b_j)$ . So there exists $j_1, j_2$ such that $a_{j_1} e b_{j_1} wa_{j_2} eb_{j_2} \neq 0$, otherwise $w=0$, a contradiction. Hence $eb_{j_1} w a_{j_2} e \neq 0$, contrary to $eYe=0$. Thus $e Re$ is ideal extending. By [L, Corollary 18.35], the ideal extending property is a Morita invariant.
\end{proof}

\begin{proposition}
Let $R$ be an $n$-by-$n$ generalized  matrix ring and $S = R^{\rm la}$.
\begin{enumerate}
	\item [(1)] If $_SX_S \leq {_S}R_S$ and $X \cap S$ is ideal essential in $eS$, for some $e \in \mathcal{S}_r(S)$, then $X$ essential in $eR$ as an $(S,S)$-bisubmodule.
	\item[(2)] If $S$ is an ideal extending ring, then for each $X \trianglelefteq R$ there is an $e \in \mathcal{B}(S)$ such that $X$ is ideal essential in $eR$.
\end{enumerate}
\end{proposition}

\begin{proof} (1) Assume $(1-e) X \neq 0$. Since $1-e \in \mathcal{S}_\ell (S), (1-e)X$ is an $(S,S)$-bisubmodule of $R$. By Theorem 2.12(3), $0 \neq (1-e) X \cap S \subseteq X \cap S \subseteq eS$, a contradiction. Then $X \subseteq eR$. Let $0 \neq \, _SY_S \leq \, _SR_S$ such that $Y \subseteq eR$ and $Y \cap X =0$. Hence $0 \neq Y \cap S \subseteq eS$ and $Y \cap S \trianglelefteq S$, a contradiction.

(2) This part is a consequence of (1).
\end{proof}

\begin{example} (1) This example illustrates Proposition 2.18(1).

Let $R = \left[ \begin{array}{cc} \Z \times \Z_4 & \Z \times 2\Z_4\\
\Z \times \{0\} & \Z \times \Z \end{array}\right]$, and let $S = R^{\rm la}$. Then $S = \left[ \begin{array}{cc} \Z \times \Z_4 & \Z \times 2\Z_4\\
\{0\} \times \{0\} & \Z \times \Z \end{array}\right]$. Take $X = \left[ \begin{array}{cc} \{0\} \times \{0\} & \{0\} \times \{0\} \\
\{0\} \times \{0\} & \{0\} \times 2 \Z \end{array}\right] \trianglelefteq R$ and $e = \left[ \begin{array}{cc} (0,0) & (0,0) \\
(0,0) & (0,1) \end{array}\right]$. Then $e \in \mathcal{S}_r(S)$ and $_SX_S$ is essential as an $(S,S)$-bisubmodule of $eR$. Note that $e \not\in \mathcal{B}(S)$.

(2) This example shows that in Proposition 2.18(2), $S$ cannot be replaced by~$D(R)$. Let $R = \left[ \begin{array}{cc} \Z_4 & 2\Z_4\\
0 & \Z_4 \end{array}\right]$ (note that $R=S=R^{\rm la}$). Then $D(R)$ is a commutative selfinjective ring, hence it is ideal extending. However $\mathcal{B}(R) = \{0,1\}$, but $\left[ \begin{array}{cc} 2\Z_4 & 0 \\
0 & 0 \end{array}\right]$ and $\left[ \begin{array}{cc} 0 & 2 \Z_4\\
0 & 0 \end{array}\right]$ are ideals of $R$ whose intersection is zero. Therefore $R$ is not ideal extending.
\end{example}

\begin{lemma}
If $A$ is an ideal extending ring, then $R$ is ideal extending where $R$ is the $n$-by-$n$ upper triangular matrix ring over $A$.
\end{lemma}

\begin{proof}
Let $X \trianglelefteq R$. Then
$$X = \left[ \begin{array}{ccccc}
X_{11} & X_{12} & \cdots & X_{1n} \\
0 & X_{22} & & X_{2n} \\
\vdots & \ddots  & \ddots & \vdots \\
0 & \cdots & 0 & X_{nn} \end{array}\right],$$
where each $X_{ij} \trianglelefteq A, X_{ii} \subseteq X_{i, i+1} \subseteq \cdots \subseteq X_{in}$ and $X_{ii} \subseteq X_{i-1, i} \subseteq  \cdots \subseteq X_{1i}$ for all $1 \leq i \leq n$. Observe $e \in \mathcal{B}(R)$ if and only if $e = c1_R$, for some $c \in \mathcal{B}(A)$. There exists $f\in \mathcal{B}(A)$ such that $X_{1n}$ is ideal essential in $fA$. Then a routine argument shows that $X$ is ideal essential in $(f1_R)R$.
\end{proof}

The following corollary is an application of Proposition 2.18 (hence of Theorem~2.12).

\begin{corollary}
Let $A$ be a ring and $R$ the $n$-by-$n$ generalized matrix ring of the form
$$R = \left[ \begin{array}{ccccc}
A & A & \cdots & A \\
X_{21} & A & \cdots & A \\
\vdots & \ddots & \ddots & \vdots\\
X_{n1} & \cdots & X_{n, n-1} & A \end{array}\right],$$
where $X_{ij} = A$ for $i \leq j,\  X_{ij} \trianglelefteq A$ for $j < i, \ X_{j1} \subseteq X_{j2} \subseteq \cdots \subseteq X_{jn}$ and $X_{ni} \subseteq X_{n-1, i} \subseteq \cdots \subseteq X_{1i}$, for all $1 \leq i \leq n$ and $1 \leq j \leq n$. Then $A$ is ideal extending if and only if $R$ is ideal extending.
\end{corollary}

\begin{proof} $(\Rightarrow)$ Assume $A$ is ideal extending. Observe that $R^{\rm la}$ is the $n$-by-$n$ upper triangular matrix ring over $A$. By Lemma 2.20, $R^{\rm la}$ is ideal extending. From Proposition 2.18(2), $R$ is ideal extending.

$(\Leftarrow)$ Assume $R$ is ideal extending. Let $X \trianglelefteq A$ and $Y$ the set of $n$-by-$n$ matrices over $X$. Then $Y \trianglelefteq R$. So there exists $e \in \mathcal{B}(R)$ such that $Y$ is ideal essential in $eR$ and $e = c1_R$ where $c \in\mathcal{B}(A)$. Then $X$ is ideal essential in $cA$. Therefore $A$ is ideal extending. 
\end{proof}

Assume $R$ is ring isomorphic to $R_m$ and to $R_n$, where $R_m$ is an $m$-by-$m$ generalized matrix ring, and $R_n$ is an $n$-by-$n$ generalized matrix ring, with $0 < m < n$. One may naturally ask:

\begin{enumerate}
	\item If $R_m \in \mathcal{T}_m$, must $R_n \in \mathcal{T}_n$?
	\item If $R_n \in \mathcal{T}_n$, must $R_m \in \mathcal{T}_m$?
\end{enumerate}

The following example shows that, in general, neither question has an affirmative answer.

\begin{example}
Let $A$ be a ring.

\begin{enumerate}
	\item Let $$R = \left[ \begin{array}{ccc} A & A & A\\
	A & A & A \\
	0 & 0 & A\end{array}\right].$$
	
	\noindent Then $R \not\in \mathcal{T}_3$, because $E_{22} \not\in \mathfrak{P}_{\rm it}(R)$ by Theorem 2.2. 
	
	Let $R_1 = \left[ \begin{array}{cc} A & A \\
	A & A \end{array}\right], \ M_{12} = \left[ \begin{array}{c} A\\ A\end{array}\right]$, $M_{21} = [0 \ 0 ]$ and $R_2 = A$. Then $R \cong \left[ \begin{array}{ll} R_1 & M_{21}\\
	M_{21} & R_2 \end{array}\right] \in \mathcal{T}_2$.
	
	\item Let
	$$R = \left[ \begin{array}{ccccc} A & A & A & A\\
	0 & A & 0 & 0 \\
	0 & A & A & A \\
	0 & A & 0 & A\end{array}\right].$$
	
	\noindent Then $R \in \mathcal{T}_4$. 
	
	Let $R_1 = \left[ \begin{array}{ll} A & A\\
	0 & A \end{array}\right], M_{12} = \left[ \begin{array}{ll} A & A\\
	0 & 0 \end{array}\right]$, $M_{21} = \left[ \begin{array}{ll} 0 & A \\
	0 & A\end{array}\right]$ and $R_2 = \left[ \begin{array}{ll} A & A\\
	0 & A\end{array}\right]$. Then $R \cong \left[ \begin{array}{cc} R_1 & M_{12} \\
	M_{21} & R_2 \end{array}\right] \not\in \mathcal{T}_2$, since $M_{12} M_{21} \neq 0$.
	
\end{enumerate}
\end{example}\smallskip

\begin{proposition}
Let $R$ be a ring with a complete set $\{e_i\}^n_{i=1} \ (n> 1)$ of orthogonal idempotents. If $\{e_i\}_{i=1}^n \subseteq \mathfrak{P}_{\rm t}(R)$, then any partition of $R^\pi$ into an $m$-by-$m$ block form is in $\mathcal{T}_m$ where $m \leq n$.
\end{proposition}

\begin{proof}
The proof is straightforward. 
\end{proof}

Note that in Proposition 2.23 if $n =3$, then we can replace $\mathfrak{P}_{\rm t}(R)$ with $\mathfrak{P}_{\rm it}(R)$.

\medskip

\noindent QUESTION: Let $R$ be a ring with a complete set $\{e_i\}^n_{i=1}  \ (n> 1)$ of orthogonal idempotents. What are necessary and sufficient conditions so that any partition of $R^\pi$ into $m$-by-$m$ block form is in $\mathcal{T}_m$ for $1 < m \leq n$?

\vskip0.5cm

\section{$n$-Peirce Rings}

\begin{definition}
A ring $R$ is called a  {\it 1-Peirce ring} if $\mathfrak{P}_{\rm t}(R) =\{0,1\}$, with $0\ne 1.$ Inductively, for a natural number $n > 1$, a ring $R$ is called an {\it $n$-Peirce} ring if there is an $e \in \mathfrak{P}_{\rm t}(R)$ such that $eRe$ is an $m$-Peirce ring for some $1 \leq m < n$ and $(1-e)R(1-e)$ is an $(n-m)$-Peirce ring.\end{definition}

\begin{example} (1)  If $R_R$ is indecomposable or $R$ is prime, then $R$ is 1-Peirce. In fact, if $R$ is semiprime then $R$ is 1-Peirce if and only if $R$ is indecomposable (as a ring).

     (2)  If $R$ is an $n$-by-$n$ generalized upper (lower)  triangular matrix ring with 1-Peirce diagonal rings, then $R$ is a $n$-Pierce ring.
		
     (3)  If $R$ has a complete set of $n$ orthogonal primitive idempotents which are Peirce trivial, then $R$ is an $n$-Peirce ring (e.g., see Example 2.5(3)).
		
     (4)  Example 2.5(2) is a 3-Peirce ring that has a complete set of three primitive idempotents which are inner Peirce trivial but not all of them are Peirce 
          trivial.
					
     (5)  Let $I$ be be an infinite index set, for each $i\in I$ let $A_i$ be a ring with only trivial idempotents, and $A = \Pi_{i\in I} A_i$.  Assume 
          $R = \left[ \begin{array}{cc} A & X\\ 0 & X\\ \end{array}\right],$ where $X$ is a nonzero ideal of $A$.  Then $R$ is in $\mathcal{T}_2$, but $R$ is not an $n$-Peirce ring for any positive integer $n$.
\end{example}               
          
 From Example 3.2(5) and Theorem 3.7, we see that the class of $n$-Peirce $n$-by-$n$ generalized matrix rings is a proper subclass of the class $\mathcal{T}_n$ for $n > 1$.  Also, due to the symmetry
          of Peirce idempotents (i.e., $e\in \mathfrak{P}_{\rm t}(R)$ if and only if $1-e\in\mathfrak{P}_{\rm t}(R)$) the class of $n$-Peirce rings exhibits better behavior than $\mathcal{T}_n$ with respect to finiteness conditions.

\begin{theorem} Let $R$ be a ring. Then either:

(1) $R$ is a 1-Peirce ring;

(2) R is an $n$-Peirce ring for $n>1$, and for each $k\in\Bbb Z^+$ with $1 < k \le n$ there exists a complete set of orthogonal idempotents, $\{e_i\}_{i=1}^k,$ such that $R^\pi \in \mathcal{T}_k$; or

(3) for each integer $k$ with $k > 1$, there exists a complete set of orthogonal idempotents, $\{e_i\}_{i=1}^k$, such that $R^\pi \in \mathcal{T}_k$. 
\end{theorem}

\begin{proof} Observe that $\{0, 1\}\subseteq \mathfrak{P}_{\rm t}(R)$, and $\{0, 1\} = \mathfrak{P}_{\rm t}(R)$ if and only if $R$ is a 1-Peirce ring.  If $\{0, 1\} \ne \mathfrak{ P}_{\rm t}(R)$ (i.e., $R$ is not 1-Peirce), then  there exists $\{e_1,1-e_1\}\subseteq \mathfrak{P}_{\rm t}(R)\setminus\{0,1\}.$ By Theorem 2.2, 
$R^\pi \in \mathcal{T}_2$. Now if at least one of $e_1Re_1$ or $(1-e_1)R(1-e_1)$ is not 1-Peirce, say $e_1Re_1$, then there exists $e_2\in\mathfrak{P}_{\rm t}(e_1Re_1)\setminus\{0,e_1\}.$ By Lemma 1.8, $\{e_2,e_1-e_2\}\subseteq \mathfrak{P}_{\rm t}(e_1Re_1) \subseteq \mathfrak{P}_{\rm it}(e_1Re_1) =  e_1Re_1 \cap  \mathfrak{P}_{\rm it}(R) \subseteq  \mathfrak{P}_{\rm it}(R).$ Hence $\{e_2,e_1-e_2,1-e_1\}$ is a complete set of orthogonal idempotents contained in $\mathfrak{P}_{\rm it}(R).$ By Theorem 2.2, 
$R^\pi \in \mathcal{T}_3$. If at least one of $e_2Re_2, \ (e_1-e_2)R(e_1-e_2),$ or $(1-e_1)R(1-e_1)$ is not 1-Peirce, then either this inductive process will terminate in $n$ steps for some $n\in\Bbb Z^+$ yielding condition (2) or it will continue indefinitely yielding condition (3).
\end{proof}

Note that in Example 2.22(2), $R$ has a 2-by-2 block form which is not in $\mathcal{T}_2$.  Observe that one can show that $E_{11}\in\mathfrak{P}_{\rm t}(R)$, so $R$ is not 1-Peirce.   Surprisingly, Theorem 3.3 predicts that there is a 2-by-2 block form for $R$ which is in $\mathcal{T}_2$.  Indeed, $R$ can be  partitioned  into another 2-by-2 block form which is in $\mathcal{T}_2$ by taking $R_1$ to be a 1-by-1 matrix and $R_2$ to be a 3-by-3 matrix (i.e., this corresponds to taking $e_1 = E_{11}$ and $1-e_1 = E_{22} + E_{33} + E_{44}$ in the proof of Theorem 3.3).

\begin{proposition}
Let $0 \neq e =e^2 \in R$ such that $eRe$ is a 1-Peirce ring, $c \in \mathfrak{P}_{\rm t}(R)$, $c_1 \in \mathfrak{P}_{\rm t}(cRc), 0\neq cec$ and $c_1 e c_1 \neq0$. Then:

\begin{enumerate}
	\item $ece = e=ec_1 e$.
	\item $cecRcec$ is a 1-Peirce ring.
\end{enumerate}
\end{proposition}

\begin{proof} (1) By Lemma 1.7, $0 \neq c ec = (cec)^2 = c(ece)c$. Thus $ece\neq 0$. From Lemma 1.8(1), $c_1 \in \mathfrak{P}_{\rm it}(R)$. So, by Lemma 1.7, $c_1ec_1 = (c_1ec_1)^2 = c_1(ec_1e)c_1$. Hence $ec_1e \neq 0$.

\smallskip

Claim 1. $ece \in \mathfrak{P}_{\rm it} (eRe)$. 

Let $x,y \in eRe$. By Lemma 1.7, $ece=(ece)^2$. Consider $(ece)x(e-ece)y(ece) =e(c[exe(1-c)eye]c)e=e0e=0$, because $c \in \mathfrak{P}_{\rm t}(R) \subseteq \mathfrak{P}_{\rm it}(R)$. Thus $ece \in \mathfrak{P}_{\rm it}(eRe)$.

\smallskip

Claim 2. $ece \in \mathfrak{P}_{\rm t}(eRe)$. 

Consider $(e-ece)xecey(e-ece) = e[(1-c)execeye(1-c)]e=e0e=0$, because $c \in \mathfrak{P}_{\rm ot}(R)$. Thus $ece \in \mathfrak{P}_{\rm ot}(eRe)$. Hence $0\neq ece \in \mathfrak{P}_{\rm t}(eRe)$. Since $eRe$ is $1$-Peirce, $ece=e$.

\smallskip

Claim 3. $ec_1e \in \mathfrak{P}_{\rm it}(eRe)$. 

From above, $c_1 \in \mathfrak{P}_{\rm it}(R)$. Let $x,y \in eRe$. Then $ec_1ex(e-ec_1e)y ec_1e= e[c_1xe(1-c_1)eyc_1]e=0$. Hence $ec_1e \in \mathfrak{P}_{\rm it} (eRe)$.

\smallskip

Claim 4. $ec_1e\in \mathfrak{P}_{\rm t}(eRe)$. 

Since $e =  ece,$ we have

$\begin{array}{lll}  (e-ec_1e)xec_1ey(e-ec_1e) & = & (ece-ec_1e)xc_1y(ece-ec_1e) \\ 
& = & e(c-c_1) exc_1ye(c-c_1)e\\ 
& = & e[(c-c_1)(cxc)c_1(cyc)(c-c_1)]e=0, \end{array}$

\noindent because $c_1 \in \mathfrak{P}_{\rm t}(cRc) \subseteq \mathfrak{P}_{\rm ot}(cRc)$. Thus $ec_1e \in \mathfrak{P}_{\rm t}(eRe)$. Since $eRe$ is a 1-Peirce ring, $ec_1e=e$.

\smallskip

(2) Let $0 \neq f \in \mathfrak{P}_{\rm t}(cecRcec)$. Observe $0 \neq f = c[ecfce]c=c[efe]c$ since $c \in \mathfrak{P}_{\rm t}(R)$. So $efe\neq 0$.

\smallskip

Claim 5. $efe= (efe)^2$.

Observe $f= (cec)f=c(cecf) = cf$. Similarly, $f=fc$. Consider $efe = e(fcecf)e=efefe=(efe)^2$.

\smallskip

Claim 6. $efe\in \mathfrak{P}_{\rm it}(eRe)$.

Let $x,y \in eRe$. By (1) $e=ece$, so 

$\begin{array}{lll}
efex(e-efe)yefe & = & efex(ece-efe) yefe= efexe(c-f)eyefe\\ 
& =& e(fc)exe(c-f)eye(cf)e\\
&= & ef[(cec)x(cec)(cec)(c-f)(cec)y(cec)]fe\\ 
&=& ef[(cec)x(cec)f(cec)(c-f)(cec)y(cec)]fe\\ 
&= & ef[(cec)x(cec)f(c-f)(cec)y(cec)]fe\\ 
& =& ef[(cec)x(cec)0(cec)y(cec)]fe\\ 
&= & 0,  \end{array}$

\noindent since $f \in \mathfrak{P}_{\rm t}(cecRcec)\subseteq \mathfrak{P}_{\rm it}(cecRcec)$.

\smallskip

Claim 7. $efe \in \mathfrak{P}_{\rm t}(eRe)$.

Consider 

$\begin{array}{lll}
(e-efe)xefey(e-efe) & =& (ece-efe)xfy(ece-efe)\\
& =& e(c-f)exfye(c-f)e\\
& =& ece(c-f)xfy(c-f)ece\\
& =& ecec(c-f)(cecxcec)f(cecycec)(c-f)ece\\
&= & e(cec-f)[(cecxcec)f(cecycec)](cec-f)e\\

&= & 0,  \end{array}$

\noindent since $f \in \mathfrak{P}_{\rm ot}(cecRcec).$ Thus $efe \in \mathfrak{P}_{\rm t}(eRe)$. Hence $efe=e$. So $0 \neq f = cecfcec=c(efe)c=cec$. Therefore $cecRcec$ is a 1-Peirce ring.
\end{proof}

Observe that in Proposition 3.4 the conclusions $ece=e$ and $cecRcec$ is a 1-Peirce ring do not need the conditions $c_1 \in \mathfrak{P}_{\rm t}(cRc)$ and $c_1ec_1 \neq 0$. Also, this result can be extended under related hypotheses (e.g. $e$ primitive and $c \in \mathfrak{P}_{\rm it}(R)$).

\begin{corollary}
Let $0 \neq e = e^2 \in R$ such that $eRe$ is a 1-Peirce ring and $c \in \mathfrak{P}_{\rm t}(R)$. The following conditions are equivalent:

\begin{enumerate}
	\item $cec\neq 0$.
	\item $ece=e$.
	\item $(1-c)e(1-c)=0$.
\end{enumerate}
\end{corollary}

\begin{proof}
(1) $\Rightarrow$ (2) This implication follows from Proposition 3.4.

(2) $\Rightarrow$ (3) Since $c \in \mathfrak{P}_{\rm t}(R)$, $e(1-c)e=ece(1-c)ece=0$. By Lemma 1.7(1), $(1-c)e(1-c)=[(1-c)e(1-c)]^2=(1-c)e(1-c)e(1-c)=0$.

(3) $\Rightarrow$ (1) Assume $(1-c)e(1-c)=0$ and $cec=0$. Then

$\begin{array}{lll} e &= & (c+(1-c)) e(c+(1-c))\\
\\
& =& cec + ce (1-c)+(1-c)ec +(1-c)e(1-c)\\
\\
& =& ce(1-c)+(1-c)ec. \end{array}$

So $e = e^2=(ce(1-c)+(1-c)ec)^2 = ce(1-c)ce(1-c)+ce(1-c)ec+(1-c)ece(1-c)+(1-c)ec(1-c)ec=0$, a contradiction.
\end{proof}

\begin{corollary}
Let $\{e_i\}^n_{i=1}$ be a complete set of nonzero orthogonal idempotents such that each $e_iRe_i$ is a 1-Peirce ring and $0 \neq c\in \mathfrak{P}_{\rm t}(R)$. Then:

\begin{enumerate}
	\item $c= \sum_{i\in J_1} ce_ic$ and $1-c = \sum_{i \in J_2}(1-c)e_i(1-c)$, where $ce_ic\neq 0$ for all $i \in J_1$ and $(1-c)e_i(1-c)\neq 0$ for all $i \in J_2$.
	\item $|J_1| + |J_2| = n$.
	\item $\{ce_ic \ \vert \ i \in J_1\} \cup \{(1-c)e_i(1-c) \ \vert \ i \in J_2\}$ is a complete set of orthogonal idempotents, where each $ce_icRce_ic$ and each $(1-c)e_i(1-c)R(1-c)e_i(1-c)$ is a 1-Peirce ring.
\end{enumerate}
\end{corollary}

\begin{proof} (1) $c = c1c=c\sum_{i\in J_1}e_ic = \sum_{i \in J_1}ce_ic$, and similarly, $1-c = \sum_{i\in J_2} (1-c)e_i(1-c)$, where $J_1 \cup J_2 \subseteq \{1, \ldots, n\}$.
	
	(2) This part follows from Corollary 3.5.
	
	(3) Note that $1 = c+1-c = \sum_{i\in J_1} ce_ic + \sum_{i\in J_2}(1-c)e_i(1-c)$. Also, $(ce_ic)ce_jc = ce_ie_j c =0$ for all $i\neq j$, since $c \in \mathfrak{P}_{\rm t}(R)$. Similarly, $[(1-c)e_i(1-c)][(1-c)e_j(1-c)]=0$ for all $i \neq j$. Moreover $[(1-c)e_i(1-c)][ce_jc] = 0= [ce_ic][(1-c)e_j(1-c)]$ for all $i, j$. By Lemma 1.7, $ce_ic$ and $(1-c)e_i(1-c)$ are idempotents for all $i$. From~Proposition~3.4(2), each $ce_i cRce_ic$ and each $(1-c)e_i(1-c)R(1-c)e_i(1-c)$ is a 1-Peirce ring.

\end{proof}

\begin{theorem} (1) If $R$ is an $n$-Peirce ring $(n> 1)$, then there is a complete set of orthogonal idempotents $\{e_i\}^n_{i=1} \subseteq \mathfrak{P}_{\rm it}(R)$ (hence $R^\pi \in \mathcal{T}_n$) such that every $e_iRe_i$ is a 1-Peirce ring.
	
	(2) If a ring $R$ has a complete set $\{e_i\}^n_{i=1}$ of orthogonal idempotents  for some $n \geq 2$ such that every $e_iRe_i$ is a 1-Peirce ring, then $R$ is a $k$-Peirce  ring for some $1 \leq k \leq n$.
\end{theorem}

\begin{proof} (1) We use strong induction on $n$. First, let $R$ be a 2-Peirce ring. Then there is an $e \in \mathfrak{P}_{\rm t}(R)$ such that $eRe$ and $(1-e)R(1-e)$ are 1-Peirce rings, and $\{e,1-e\}$ is a complete set of orthogonal idempotents, with $e, 1-e\in \mathfrak{P}_{\rm it}(R)$.
	
	Next, consider a fixed $n \geq 2$ and assume that for each $k, \ 2 \leq k \leq n$, if $R$ is a $k$-Peirce ring, then there is a complete set of orthogonal idempotents $\{e_i\}^k_{i=1} \subseteq \mathfrak{P}_{\rm it}(R)$ such that every $e_iRe_i$ is a 1-Peirce ring. Now let $R$ be an $(n+1)$-Peirce ring. Then there is a $c \in \mathfrak{P}_{\rm t}(R)$ such that $cRc$ is a $k$-Peirce ring for some $k, 1\leq k < n+1$, and $(1-c)R(1-c)$ is an $(n+1-k)$-Peirce ring. Since $k \leq n$ and $n+1-k \leq n$, and assuming for the moment that $2 \leq k$ and $2 \leq n +1-k$, the induction hypothesis guarantees the existence of complete sets of  orthogonal idempotents $\{e_i\}^k_{i=1} \subseteq \mathfrak{P}_{\rm it}(cRc)$ and $\{f_j\}^{n+1-k}_{j=1} \subseteq \mathfrak{P}_{\rm it} ((1-c)R(1-c))$, such that $e_icRce_i$ and $f_j(1-c)R(1-c)f_j$ are 1-Peirce rings for every $i$ and $j$. Since $c, 1-c\in \mathfrak{P}_{\rm it}(R)$, it follows from Lemma 1.8 that $\{e_i\}^k_{i=1}, \{f_j\}^{n+1-k}_{j=1} \subseteq \mathfrak{P}_{\rm it}(R)$. Since $(cRc)((1-c)R(1-c))=0$, we conclude that $\{e_i\}^k_{i=1} \cup \{f_j\}^{n+1-k}_{j=1}$ is an orthogonal set of idempotents in $\mathfrak{P}_{\rm it}(R)$. Moreover, $\sum_{i=1}^ke_i + \sum_{j=1}^{n+1-k}f_j = c+(1-c)=1$, and $e_i cRce_i = e_iRe_i$ (since $e_i \in cRc)$ and $f_j(1-c)R(1-c)f_j= f_jRf_j$ for every $i$ and~$j$.
	
	Finally, we consider the case $k=1$ or $n+1-k=1$. Without loss of generality, let $k=1$, i.e., $c \in \mathfrak{P}_{\rm it}(R), cRc$ is a 1-Peirce ring and $(1-c)R(1-c)$ is an $n$-Peirce ring. Then we can proceed as in the previous paragraph with $c \in \mathfrak{P}_{\rm it}(R)$ and a complete set of orthogonal idempotents $\{f_j\}^n_{j=1}  \subseteq \mathfrak{P}_{\rm it}((1-c)R(1-c))$, and then the set $\{c, f_1, \ldots, f_n\}$ is a complete set of orthogonal idempotents in $\mathfrak{P}_{\rm it}(R)$ such that $cRc$ and $f_jRf_j$ are 1-Peirce rings for all $j$.
	
	(2) We again use strong induction on $n$. Let $R$ have a complete set of orthogonal idempotents $\{e_1, e_2\}$ such that each $e_iRe_i$ is a 1-Peirce ring. If $R$ is a 1-Peirce ring, then we are done, since $1\leq 2$. Otherwise there is a $c \in \mathfrak{P}_{\rm t}(R)$ such that $c \not\in \{0,1\}$. Hence $1-c\in \mathfrak{P}_{\rm t}(R)$ and $1-c\not\in \{0,1\}$. From Corollary 3.6, $c = ce_ic$ for $i \in \{1,2\}$. Without loss of generality, assume $i =1$. Then, again by Corollary~3.6, $cRc= ce_1cRce_1c, (1-c)R(1-c)=(1-c)e_2(1-c)R(1-c)e_2(1-c)$, and $cRc$ and $(1-c)R(1-c)$ are 1-Peirce rings. Therefore $R$ is a 2-Peirce ring.
	
	Next assume that the result holds for a fixed $n \geq 2$. Let $R$ be a ring having a complete set of orthogonal idempotents $\{e_i\}^{n+1}_{i=1}$ such that each $e_iRe_i$ is a 1-Peirce ring. If $R$ is a 1-Peirce ring, we are done. Otherwise there is a $c \in \mathfrak{P}_{\rm t}(R)$ such that $c \not\in \{0,1\}$. Hence $1-c \in \mathfrak{P}_{\rm t}(R)$ and $1-c \not\in\{0,1\}$. From Corollary 3.6, there exist $J_1, J_2 \subseteq \{1, \ldots, n\}$ and complete sets of orthogonal idempotents $\{ce_i c \ \vert \ i \in  J_1\}$ and $\{(1-c)e_i(1-c) \ \vert \  i \in J_2\}$ for $cRc$ and $(1-c)R(1-c)$, respectively. From the induction hypothesis, there exist positive integers $k_1$ and $k_2$ such that $1 \leq k_1 \leq |J_1|$ and $1\leq k_2 \leq |J_2|$ such that $cRc$ is a $k_1$-Peirce ring and $(1-c)R(1-c)$ is a $k_2$-Peirce ring. Since $|J_1| + |J_2| =n+1$, then $R$ is $k$-Peirce where $k=k_1+k_2$ and $1 \leq k \leq n+1$.
\end{proof}

\begin{corollary}
Let $\{e_i\}^n_{i=1}\subseteq \mathfrak{P}_{\rm it}(R)$ be a complete set of orthogonal idempotents such that each $e_iRe_i$ is a $k_i$-Peirce ring for some positive integers $k_i$. Then $R$ is a $k$-Peirce ring for some $1\le k\le \sum_{i=1}^nk_i.$
\end{corollary}

\begin{proof}
This result follows from Theorem 3.7.
\end{proof}

From Theorem 3.7, $R$ is an $n$-Peirce ($n > 1$) generalized matrix ring implies that $R\in T_n$; and if $R$ has a complete set
     of $n$ orthogonal primitive idempotents, then $R$ is $k$-Peirce for some $k$ with $1 \le k \le n$.  Thus it is natural to ask:  If $R$ has a complete set of orthogonal 
     primitive idempotents $\{e_i\}_{i=1}^n \subseteq \mathfrak{P}_{\rm it}(R)$ (hence $R^\pi\in \mathcal{T}_n$), must $R$ be $n$-Peirce? Observe that for n = 2, the question has an affirmative answer.  Our next example provides a negative answer, in general.

\begin{example}
Let
$$R = \left[ \begin{array}{ccc} A & X^2 & X\\
X & A & X^2 \\
X^2 & X & A \end{array}\right],$$
where $A$ is a ring such that $0$ and $1$ are the only idempotents of $A$; and $0 \neq X \trianglelefteq A$ such that $X \neq X^2, X^2 \neq X^3$, and $X^3 =0$. Then $R \in \mathcal{T}_3$, but $R$ is a 1-Peirce ring. One can construct such rings by letting $B$ be a commutative ring, $0 \neq P$ a prime ideal of $B$ such that $P \neq P^2$ and $P^2 \neq P^3$. Then take $A = B / P^3$ and $X = P/P^3$. In particular, let $B = F[x]$ where $F$ is a field and $P = xF[x]$. 

Since $X^2 X = XX^2=0$, Corollary 2.3 yields $R \in \mathcal{T}_3$. To show that $R$ is a 1-Peirce ring, we first characterize all nontrivial idempotents of $R$. Let $\alpha \in R$ such that $\alpha \neq 0$ and $\alpha \neq 1$. Then $\alpha = \alpha^2$ if and only if
$$\alpha = \left[ \begin{array}{ccc} e_1 & m_{12} & m_{13}\\
m_{21} & e_2 & m_{23}\\
m_{31} & m_{32} & e_3 \end{array}\right]$$
where $m_{12}, m_{23}, m_{31} \in X^2, \ m_{13}, m_{21}, m_{32} \in X, e_i \in \{0,1\}$, and the following equations are satisfied:

$$\begin{array}{lll}
e_1m_{12} + m_{12} e_2 + m_{13} m_{32} & =& m_{12}\\
e_1 m_{13} + m_{13} e_3 & =& m_{13}\\
m_{21} e_1 + e_2m_{21} & =& m_{21}\\
e_2m_{23} + m_{23}e_3 + m_{21} m_{13} & =& m_{23}\\
m_{31} e_1 + e_3m_{31} + m_{32} m_{21} & =& m_{31}\\
m_{32} e_2 + e_3m_{23} & =& m_{32}. \end{array}$$
\end{example}

From the above conditions $\alpha$ must have one of the following six forms:
\begin{enumerate}
	\item [(i)] $\left[ \begin{array}{ccc} 1& m_{12} & m_{13}\\
	m_{21} & 0 & m_{23}\\
	m_{31} & 0 & 0 \end{array}\right]$ with $m_{21} m_{13} = m_{23}$;
	
	\vskip .4cm
	
	\item[(ii)] $\left[ \begin{array}{ccc} 0 & m_{12} & 0 \\
	m_{21} & 1 & m_{23}\\
	m_{31} & m_{32} & 0 \end{array}\right]$ with $m_{32} m_{21} = m_{31}$;
	
	\vskip .4cm
	
	\item[(iii)] $\left[ \begin{array}{ccc} 0 & m_{12} & m_{13} \\
	0 & 0 & m_{23} \\
	m_{31} & m_{32} & 1 \end{array}\right]$ with $m_{13} m_{32} = m_{12}$;
	
	\vskip .4cm
	
	\item[(iv)] $\left[ \begin{array}{ccc} 1 & m_{12} & m_{13}\\
	0  & 1 & m_{23} \\
	m_{31} & m_{32}  & 0 \end{array}\right]$ with $m_{13} m_{32} = -m_{12}$;
	
	\vskip .4cm
	
	\item[(v)] $\left[ \begin{array}{ccc} 1 & m_{12} & 0 \\
	m_{21} & 0 & m_{23} \\
	m_{31} & m_{32} & 1\end{array}\right]$ with $m_{32} m_{21} = -m_{31}$;
	
	\vskip .4cm
	
	\item[(vi)] $\left[ \begin{array}{ccc} 0 & m_{12} & m_{13} \\
	m_{21} & 1 & m_{23}\\
	m_{31} & 0 & 1 \end{array}\right]$ with $m_{21} m_{13} = -m_{23}$.
\end{enumerate}

Now assume $R$ is not a 1-Peirce ring. Then there exist $c, 1-c \in \mathfrak{P}_{\rm t}(R)$ such that $0 \neq c$ and $c \neq 1$.

Then either $c$ has a form of type (i), (ii), or (iii); or $1-c$ has such a form. Without loss of generality, assume $c$ has a form of type (i), (ii) or (iii). Then $1-c$ has a form of type (iv), (v), or (vi). We show that no matrix of type (iv), (v) or~(vi) is in $\mathfrak{P}_{\rm ot}(R)$. Hence $1-c \not\in \mathfrak{P}_{\rm t}(R)$, a contradition. Therefore $R$ is a 1-Peirce ring.

Since $X^2\neq 0$, there exist $x,y \in X$ such that $0 \neq xy$. Observe:

\smallskip

$\left[ \begin{array}{ccc}
0 & m_{12} & m_{13}\\
m_{21} & 1 & m_{23}\\
m_{31} & 0 & 1 \end{array}\right] xE_{21}  \left[ \begin{array}{ccc} 1 & -m_{12} & -m_{13}\\
-m_{21} & 0 & -m_{23}\\
-m_{31} & 0 & 0 \end{array}\right]yE_{13}  \left[ \begin{array}{ccc} 0 & m_{12} & m_{13}\\
m_{21} & 1 & m_{23}\\
m_{31} & 0 & 1\end{array}\right]$

\vskip .2cm

$= xyE_{23} \neq 0$;

\vskip .4cm

$\left[ \begin{array}{ccc} 1 & m_{12} & 0 \\
m_{21} & 0 & m_{23} \\
m_{31} & m_{32} & 1 \end{array}\right] xE_{32}\left[ \begin{array}{ccc} 0 & -m_{12} & 0 \\
-m_{21} & 1 & -m_{23}\\
-m_{31} & -m_{32} & 0 \end{array}\right] yE_{21} \left[ \begin{array}{ccc} 1 & m_{12} & 0 \\
m_{12} &0 & m_{23}\\
m_{13} & m_{32} & 1 \end{array}\right]$

\vskip .2cm

$= xyE_{31}  \neq 0$;

\vskip .4cm

$\left[ \begin{array}{ccc} 1 & m_{12} & m_{13} \\
0 & 1 & m_{23} \\
m_{31} & m_{32} & 0 \end{array}\right] xE_{13} \left[ \begin{array}{ccc} 0 & -m_{12} & -m_{13} \\
0 & 0 & -m_{23}\\
-m_{31} & -m_{32} & 1 \end{array}\right] yE_{32} \left[ \begin{array}{ccc} 1 & m_{12} & m_{13} \\
0 & 1 & m_{23}\\
m_{31} & m_{32} & 0 \end{array}\right]$

\vskip .2cm

$= xyE_{12} \neq 0$.

\begin{lemma}
Let $0 \neq c = c^2 = R$ and $e \in \mathfrak{P}_{\rm it} (cRc)$ such that $e\neq c$. Then $ReR \subsetneq RcR$.
\end{lemma}

\begin{proof}
Observe that $\{e, c-e\}$ is set of orthogonal idempotents. Clearly, $ReR \subseteq RcR$. Assume that $ReR = RcR$. Then $c = \sum r_ies_i$. So $c-e = (\sum r_i e s_i)-e$. Then
$$\begin{array}{lll}
 c-e = (c-e)^2  & = &[(\sum r_i es_i) -e] (c-e)
 \\
 & = & (\sum r_i es_i)(c-e) 
\\
& = & \sum r_i e s_i(c-e)
 \\
 &= & (\sum r_i es_i(c-e))^2
\\
 & = & \sum_j \sum_i r_i es_i(c-e)r_jes_j(c-e)
 \\
 & =& 0, \end{array}$$
\noindent since 
$$\begin{array}{lll} r_i[es_i(c-e)r_j e]s_j(c-e) & = & r_i [es_ie(c-e)r_je]s_j(c-e)
\\ 
& = & r_i 0 s_j(c-e)
\\
& = &0, \end{array}$$
\noindent because $e \in \mathfrak{P}_{\rm it}(cRc)$. However, this is a contradiction, since $e \neq c$. Therefore $ReR \subsetneq RcR$.
\end{proof}

\begin{theorem}
Assume $R$ has DCC on $\{ReR \ \vert \ e \in \mathfrak{P}_{\rm t}(R)\}$. Then $R$ is an $n$-Peirce ring for some $n \in \Z^+$.
\end{theorem}

\begin{proof}
Assume $R$ has DCC on $\{ReR \ \vert \  \ e \subseteq \mathfrak{P}_{\rm t}(R)\}$, but $R$ is a not an $n$-Peirce ring for any $n\in \Z^+$. Observe that $\mathfrak{P}_{\rm t}(R) \neq \{0,1\}$. So let $0 \neq c_1 \in \mathfrak{P}_{\rm t}(R)$ be such that $c_1 \neq 1$. Then $c_1Rc_1$ is not an $n$-Peirce ring for any $n \in \Z^+$, or $(1-c_1)R(1-c_1)$ is not an $n$-Peirce ring for any $n \in \Z^+$. Without loss of generality, say $c_1Rc_1$ is not an $n$-Peirce ring for any $n \in \Z^+$. Then $\mathfrak{P}_{\rm t}(c_1Rc_1) \neq \{0,c_1\}$. So let $0 \neq c_2\in \mathfrak{P}_{\rm t}(c_1Rc_1)$ be such that $c_2 \neq c_1$. Then  $c_2Rc_2$ is not an $n$-Peirce ring for any $n \in \Z^+$, or $(c_1- c_2)R(c_1-c_2)$ is not an $n$-Peirce ring for any $n \in \Z^+$. Without loss of generality, say $c_2Rc_2$ is not an $n$-Peirce ring for any $n \in \Z^+$. By Lemma~1.8, $c_1, c_2 \in \mathfrak{P}_{\rm it}(R)$. From Lemma 3.10, $R \supsetneq Rc_1R\subsetneq Rc_2R$. We can continue this process indefinitely, which contradicts the DCC on $\{R e R \ \vert \  e \in \mathfrak{P}_{\rm it}(R)\}$. Therefore $R$ is an $n$-Peirce ring for some $n \in \Z^+$.
\end{proof}

\begin{proposition}
If $\{b_1, \ldots, b_n\} \subseteq \mathfrak{P}_{\rm it}(R)$ is a complete set of nonzero orthogonal  idempotents such that $\mathfrak{P}_{\rm it}(b_iRb_i) = \{0,b_i\}$, then $|\{ReR \ \vert  \ e \in \mathfrak{P}_{\rm it}(R)\} | \leq 2^n$.
\end{proposition}

\begin{proof}
Let $0 \neq e = e^2 \in \mathfrak{P}_{\rm it}(R)$. By Lemma 1.8, each $b_i eb_i \in \mathfrak{P}_{\rm it}(R)$. Observe that $e = \left(\sum_{i=1}^nb_i\right) e \left(\sum_{i=1}^n b_i\right) = \sum_{i=1}^n \sum_{j=1}^n b_ieb_j$, and $b_i e b_i \neq 0$ if and only if $b_ieb_i = b_i$. Let $J \subseteq \{1, \ldots, n\}$ such that $b_i e b_i \neq 0$ if and only if $i \in J$. Since $e \in \mathfrak{P}_{\rm it}(R), e = e(e)e = e \left( \sum_{i=1}^n \sum_{j=1}^n b_i e b_j\right) e = \sum_{i=1}^n\sum_{j=1}^n eb_ieb_je = \sum_{i=1}^n \sum_{j=1} eb_ib_je= \sum_{i\in J} e b_i e$. Then $ReR \subseteq \sum_{i\in J} Rb_iR$.

Observe that $e- \sum_{i\neq j} b_i eb_j = \sum_{i\in J} b_ieb_i = \sum_{i\in J}b_i$. Let $k \in J$, then $b_ke-b_keb_j = b_k$. Hence $Rb_kR \subseteq ReR$. So $\sum_{i\in J}Rb_iR \subseteq ReR$. Therefore $ReR = \sum_{i \in J} Rb_iR$. Since $J \subseteq \{1, \ldots, n\}, \ |\{ReR \ \vert \  e \in \mathfrak{P}_{\rm it}(R)\}|\leq |\{\sum_{i\in K} Rb_iR \ \vert \  K \subseteq \{1, \ldots, n\}\}| =| \{K \ \vert \ K\subseteq \{1, \ldots, n\}\}| = 2^n$, where $\sum_{i\in K} R b_i R$ corresponds to~$\{0\}$ when $K=\emptyset$. 
\end{proof}

As an illustration and application of several of our previous results, we provide the following lemma and proposition.  Recall that a ring R is ''quasi-Baer'' if for each $X \trianglelefteq R$ there is an $e = e^2 \in R$ such that $\ul{r}_R(X) = eR$.  See [BPR2] and [C] for further details on the class of quasi-Baer rings.

\begin{lemma} $R$ is a prime ring if and only if $R$ is quasi-Baer and a 1-Peirce ring.\end{lemma}

\begin{proof} Assume $R$ is prime.  From [BHKP, Lemma 4.2] or [BPR2, Proposition~3.2.5], $R$ is quasi-Baer.  From Corollary 1.6, $R$ is a 1-Peirce ring.  Conversely, assume $xRy = 0$ for some $x, y \in R$ with $x\ne0.$  Then $y\in \ul{r}_R(xR) = \ul{r}_R(RxR) = eR$ for some $e = e^2 \in R$.  Since $\ul{r}_R(xR)$ is an ideal of $R$, then $e\in \mathcal{S}_{\ell}(R)$.  By Proposition~1.10(1), $e = 0$.  Hence $y = 0$, so R is prime.\end{proof}      

\begin{proposition}  Assume that $R$ is a quasi-Baer ring.  If $\{e_1, \ldots, e_n\}$ is a complete set of orthogonal inner Peirce trivial idempotents and each $e_iRe_i$ is a 1-Peirce ring, then $R$ is a $k$-Peirce ring for some $1\le k\le n$, $R^\pi \in \mathcal{T}_n$ and each $e_iRe_i$ is a prime ring.\end{proposition}

\begin{proof} This proof follows from Theorem 3.7(2), Theorem 2.2, Lemma 3.13, and [C, Lemma 2].\end{proof}

For example, any quasi-Baer ring with a complete set of orthogonal primitive idempotents (e.g., a right hereditary right Noetherian ring) satisfies the hypothesis of Proposition 3.14.

  In a sequel to this paper, we further investigate the properties and structure of the class of n-Peirce rings.”

\vskip .4cm

\noindent {\bf Acknowledgements:}

\smallskip 

\noindent 1. Proposition 2.8 was suggested by Arturo Magidin.

\noindent 2. Our attention was drawn to [P] by Lance Small.

\noindent 3. The first author was supported by the Hungarian National Foundation for Scientific Research under
Grants no.~K-101515 and Institute of Mathematics, Hanoi, Vietnam. 

\noindent 4. The third author was supported by the
National Research Foundation (of South Africa) under grant no.~UID 72375. Any
opinion, findings and conclusions or recommendations expressed in this
material are those of the authors and therefore the National Research
Foundation does not accept any liability in regard thereto.

\section*{References}
\begin{enumerate}
\item[{[ABP]}] E.P. Armendariz, G.F. Birkenmeier and J.K. Park, Ideal intrinsic extensions with connections to PI-rings, J. Pure Appl.~Algebra 213 (2009), 1756-1776. Corrigendum 215 (2011), 99-100.
\item[{[AvW1]}]  P.N. \'Anh and L. van Wyk, Automorphism groups of generalized triangular matrix rings, Linear Algebra Appl.~434 (2011), 1018-1025.
\item[{[AvW2]}] P.N. \'Anh and L. van Wyk, Isomorphisms between strongly triangular matrix rings, Linear Algebra Appl.~438 (2013), 4374-4381.
\item[{[BHKP]}] G.F. Birkenmeier, H.E. Heatherly, J.Y. Kim and J.K. Park, Triangular matrix representations, 
J. Algebra 230 (2000), 558-595.
\item[{[BMR]}] G.F. Birkenmeier, B.J. M\"{u}ller and S.T. Rizvi, Modules in which every fully invariant submodule is essential in a direct summand, Comm.~Algebra 30 (2002), 1395-1415.
\item[{[BPR1]}] G.F. Birkenmeier, J.K. Park and S.T. Rizvi, Hulls of semiprime rings with applications to $C^\ast$-algebras, J. Algebra 322 (2009), 327-352.
\item[{[BPR2]}] G. Birkenmeier, J.K. Park, S.T. Rizvi, Extensions of Rings and Modules, Birkh\"auser/Springer, New York, 2013.
\item[{[C]}] W.E. Clark, Twisted matrix units semigroup algebras, Duke Math.~J.~34 (1967), 417-423.
\item[{[CR]}] F. Cedo and L.H. Rowen, Addendum to ``Examples of semiperfect rings'', Israel J.~Math.~107 (1998), 343-348.
\item[{[DHSW]}] N.V. Dung, D.V. Huynh, P.F. Smith and R. Wisbauer, Extending Modules, Longman, Harlow, 1994.
\item[{[FS]}] J.W. Fisher and R.L. Snider, On the Von Neumann regularity of rings with regular prime factor rings, Pacific J.~Math.~54 (1974), 135-144.
\item[{[GW]}] B.J. Gardner and R. Wiegandt, Radical Theory of Rings, Marcel Dekker, New York, 2004.
\item[{[K]}] I. Kaplansky, Rings of Operators, Benjamin, New York, 1968.
\item[{[KT]}] P.A. Krylov, A.A. Tuganbaev, Modules over formal matrix rings (Russian), Fundam. Prikl. Mat. 15(8) (2009), 145-211; translation in 
J.~Math.~Sci.~(N. Y.) 171(2) (2010), 248-295.
\item[{[L]}] T.Y. Lam, Lectures on Modules and Rings, Springer, New York, 1999.
\item[{[P]}] B. Peirce, Linear Associative Algebra, American J.~Math. 4 (1881), 97-229.
\item[{[R1]}] L.H. Rowen, Ring Theory I, Academic Press, Boston, 1988.
\item[{[R2]}] L.H. Rowen, Examples of semiperfect rings, Israel J.~Math.~65 (1989), 273-283.
\item[{[S]}] J. Szigeti, Linear algebra in lattices and nilpotent endomorphisms of semisimple modules, J. Algebra 319 (2008), 296-308.
\item[{[TLZ]}] G. Tang, C. Li and Y. Zhou, Study of Morita contexts, Comm.~Algebra 42 (2014), 1668-1681.
\item[{[W]}] R. Wisbauer, Modules and Algebras: Bimodule Structure and Group Actions on Algebras, Addison Wesley Longman, Harlow, 1996.
\end{enumerate}



\end{document}